\documentclass[12pt, reqno]{amsart}
\usepackage{amsmath, amsthm, amscd, amsfonts, amssymb, graphicx, color, mathrsfs}
\usepackage[bookmarksnumbered, colorlinks, plainpages]{hyperref}
\usepackage[all]{xy}
\usepackage{slashed}
\usepackage{tipa}
\usepackage{mathabx}
\usepackage{soul}
\usepackage{cancel}
\usepackage{ulem}
\usepackage{inputenc}
\usepackage{booktabs}

\newtheorem{theorem}{Theorem}[section]
\newtheorem{lemma}[theorem]{Lemma}

\newtheorem{proposition}[theorem]{Proposition}
\newtheorem{corollary}[theorem]{Corollary}
\theoremstyle{definition}
\newtheorem{definition}[theorem]{Definition}

\newtheorem{remark}[theorem]{Remark}
\numberwithin{equation}{section}

\begin{document}
\setcounter{page}{1}

\title[ Subelliptic sharp G\r{a}rding inequality ]{ Subelliptic sharp G\r{a}rding inequality  on compact Lie groups}

\author[D. Cardona]{Duv\'an Cardona}
\address{
  Duv\'an Cardona:
  \endgraf
  Department of Mathematics: Analysis, Logic and Discrete Mathematics
  \endgraf
  Ghent University, Belgium
  \endgraf
  {\it E-mail address} {\rm duvan.cardonasanchez@ugent.be}
  }
  
 \author[S. Federico]{Serena Federico}
\address{
  Serena Federico:
  \endgraf
  Department of Mathematics: Analysis, Logic and Discrete Mathematics
  \endgraf
  Ghent University, Belgium
  \endgraf
  {\it E-mail address} {\rm serena.federico@ugent.be}
  }

\author[M. Ruzhansky]{Michael Ruzhansky}
\address{
  Michael Ruzhansky:
  \endgraf
  Department of Mathematics: Analysis, Logic and Discrete Mathematics
  \endgraf
  Ghent University, Belgium
  \endgraf
 and
  \endgraf
  School of Mathematical Sciences
  \endgraf
  Queen Mary University of London
  \endgraf
  United Kingdom
  \endgraf
  {\it E-mail address} {\rm michael.ruzhansky@ugent.be, m.ruzhansky@qmul.ac.uk}
  }



\thanks{The authors are supported  by the FWO  Odysseus  1  grant  G.0H94.18N:  Analysis  and  Partial Differential Equations and by the Methusalem programme of the Ghent University Special Research Fund (BOF)
(Grant number 01M01021). Michael Ruzhansky is also supported  by the EPSRC grant 
EP/R003025/2.\\
This project has received funding from the European Union’s Horizon 2020 research and the innovation programme under the Marie Sk\l odowska-Curie grant agreement No 838661.}

\begin{abstract}
In this work  we  establish a subelliptic sharp G\r{a}rding inequality on compact Lie groups for pseudo-differential operators with symbols belonging to global subelliptic H\"ormander classes.
In order for the inequality to hold we require the global matrix-valued symbol to satisfy the suitable classical nonnegativity condition in our setting. Our result extends to $\mathscr{S}^m_{\rho,\delta}(G)$-classes, $0\leq \delta<\rho$, the one in \cite{RuzhanskyTurunen2011} about the validity of the sharp G\r{a}rding inequality for the class $\mathscr{S}^m_{1,0}(G)$. We remark that the result we prove here is already new and sharp in the case of the torus.

\end{abstract} \maketitle
\allowdisplaybreaks
\tableofcontents
\section{Introduction}
\subsection{Outline and historical remarks}
In this work we establish the sharp G\r{a}rding inequality for pseudo-differential operators with symbols in the global subelliptic H\"ormander classes on compact Lie groups \cite{CardonaRuzhanskyC}. 
As a byproduct we obtain the extension to the global H\"ormander classes $\mathscr{S}^m_{\rho,\delta}(G),$ for all $0\leq \delta<\rho\leq 1$, of the sharp G\r{a}rding inequality proved in \cite{RuzhanskyTurunen2011} for the Kohn-Nirenberg  classes $\mathscr{S}^m_{1,0}(G)$.

Before describing in detail the main result of this paper concerning operators on compact Lie groups, let us briefly go back to the Euclidean case and describe the celebrated microlocal result which inspired the analysis of the problem we consider here.

G\r{a}rding's type inequalities have played a crucial role in the study of several problems related with partial and pseudo-differential operators. These inequalities are $L^2$-lower bounds which can be applied, in different contexts, to obtain results about the existence and uniqueness of solutions of differential and pseudo-differential equations. The starting point in the investigation of these fundamental lower bounds was the celebrated work of G\r{a}rding \cite{Garding1953} in which he proved the so called G\r{a}rding inequality for elliptic operators:\\
 
{\it Let $P$ be an elliptic self-adjoint pseudo-differential operator of order $m$ on an open set $\Omega\subset\mathbb{R}^n$, then, for any $\mu<m/2$ and any compact $K\subset\Omega$, there exist two positive constants $c_{\mu,K}$ and $C_{\mu,K}$ such that}\footnote{We denote by $H^s$ the standard Sobolev space of order $s$ defined as the completion of $C^\infty_0(\mathbb{R}^n)$ with respect to the norm $\Vert u\Vert_{H^s}:=\Vert (1-\Delta)^{\frac{s}{2}} u\Vert_{L^2}$, where $\Delta$ is the standard Laplacian on $\mathbb{R}^n$.}
\begin{equation}\label{GI}
      (Pu,u)\geq c_{\mu,K} \| u\|^2_{H^\frac{m}{2}}-C_{\mu,K}\|u\|_{H^{\mu}}^2,\quad \forall u\in C_0^\infty(K).
\end{equation}

Inequality \eqref{GI} was used by G\r{a}rding to derive the existence of solutions of the Dirichlet problem for elliptic operators as well as to study the distribution of the eigenvalues. However, in order to deal with non-elliptic problems, some refinements of the latter are needed. 
In particular, H\"ormander proved in \cite{Hormander1966} the following {\it Sharp G\r{a}rding inequality} for operators with symbols having nonegative real part:\\

{\it Let $P$ be a pseudo-differential operator of order $m$ defined on an open set $\Omega\subset\mathbb{R}^n$, and let $p\in S^m_{1,0}(\Omega)$ be its symbol. If $\mathsf{Re}(p(x,\xi))\geq 0$ for all $(x,\xi)\in T^*\Omega\setminus0$, then for any compact subset $K\subset \Omega$ there exists a constant $C_K>0$ such that}
\begin{equation}\label{SGI}\mathsf{Re}(Pu,u)\geq -C_K\| u\|^2_{\frac{m-1}{2}},\quad \forall u\in C_0^\infty(K).
\end{equation}
After \eqref{SGI} some generalizations were proved, that is, specifically, the suitable version for operators with symbols in $S^m_{\rho,\delta}(\mathbb{R}^n)$, and the corresponding version for systems (see, for instance, \cite{KG}).

Note that the previous inequalities are both based on a sign property of the symbol, and that no geometric property of the characteristic set is taken into account.

Further improvements of \eqref{SGI} have been established by means of a deeper analysis of the geometry of the characteristic set and of the invariants associated with the operators, like, for instance, the principal and the subprincipal symbol. In this direction we have the Melin inequality proved in \cite{Melin71} and the H\"ormander inequality proved in \cite{Hormander1977}, the latter improved by Parenti and Parmeggiani in \cite{Parenti-Parmeggiani2000} (see also \cite{Parenti-Parmeggiani1998}). A somehow different approach was adopted by Fefferman and Phong in \cite{FeffermanPhong78} where they derived the sharpest result only by requiring the nonnegativity of the {\it total symbol} of the operator.
For a survey about the fundamental lower bounds mentioned above we refer to \cite{Parmeggiani2018}.

Let us stress that these refinements not only allow the study of nonelliptic operators, but can also be used to obtain microlocal energy estimates leading to results on propagation of singularities (see H\"ormander \cite{HormanderNP}). Additionally, the (sharp) G\r{a}rding inequality and its generalizations become a fundamental tool to analyze the existence of solutions of a wide class of boundary value problems (like the $\overline{\delta}$-Neumann problem), and to investigate the global solvability of evolution problems and the local well-posedness of the Cauchy problem for evolution equations.

Let us also mention that some results, both positive and negative, about the validity of some of the fundamental lower bounds mentioned above are known for systems as well, and we refer the interested reader to \cite{Parmeggiani2004, Parmeggiani2013} for an overview of this topic. In the setting of compact Lie groups the validity of such estimates is more delicate to analyze.   Indeed, only some of the lower bounds presented above have been proved so far.
For instance, in the case of a compact Lie group $G$, and more generally on manifolds, \eqref{GI} remains valid for operators with symbols in the usual $S^m_{\rho,\delta}(M)$ H\"ormander classes, where  $0\leq \delta <\rho\leq 1,$ and $\rho\geq 1-\delta$.

We recall that on closed manifolds $(\rho,\delta)$-pseudo-differential operators are well defined provided that  $0\leq \delta <\rho\leq 1,$ $\rho\geq 1-\delta,$ and that, under these assumptions, the global classes $\mathscr{S}^m_{\rho,\delta}(G)$ defined in \cite{Ruz} coincide with the usual H\"ormander classes on compact manifolds. However bear in mind that the $\mathscr{S}^m_{\rho,\delta}(G)$-classes are defined for all $0\leq \delta <\rho\leq 1$ without the restriction $\rho\geq 1-\delta$.

On compact Lie groups, \eqref{GI} for global symbols in the classes $\mathscr{S}^m_{\rho,\delta}(G),$ for $0\leq \delta <\rho\leq 1$ (the whole range), was proved by the third author and J. Wirth in \cite{RuzhanskyWirth2014} and extended for subelliptic classes of pseudo-differential operators by the first and third author in \cite[Page 96]{CardonaRuzhanskyC}. We refer the reader to \cite[Page 27]{CKRT2021I} for G\r{a}rding type inequalities on smooth manifolds, with or without  boundary, using global symbol criteria.

As for the sharp G\"arding inequality in the general manifold setting, since a condition on the whole symbol (and not only on the principal symbol) is needed, this makes the result far reaching in this generality (symbols are not invariantly defined on manifolds, while the principal symbol is).
However, by using the description of H\"ormander classes $\Psi^m_{\rho,\delta}(G,\mathrm{loc}),$  $0\leq \delta <\rho\leq 1,$ $\rho\geq 1-\delta,$ in terms of global symbols defined on the phase space $G\times \widehat{G}$ \footnote{Here, $\widehat{G}$ is the unitary dual of $G$ which consists of all equivalence classes $[\xi]$ of irreducible, unitary continuous representations of $G,$ $\xi\in\textnormal{End}(G,H_\xi)$ on a finite dimensional vector space $H_\xi\cong \mathbb{C}^{d_\xi}.$} (see \eqref{EQequivalence} for details about these classes of operators),  the third author and Turunen proved in \cite{RuzhanskyTurunen2011} the following sharp G\"arding inequality on any compact Lie group:
\medskip

{\it Let $A\in \Psi^{m}_{1,0}(G,\mathrm{loc})$ be such that the almost positivity condition $a(x,[\xi])\geq 0$ on the global (matrix-valued) symbol of $A$ holds true, then}
\begin{equation}\label{SGIRGLapl}
    \mathsf{Re}(Au,u)_{L^2(G)}\geq -C\Vert u\Vert_{H^{\frac{m-1}{2}}(G)},\quad \forall u\in C^\infty(G).
\end{equation} 

We remark again that, in contrast with \eqref{GI}, this result requires a condition on the global symbol of $A$. This is a nontrivial difference since results involving conditions on the principal symbol only can be easily extended to manifolds, while results requiring conditions on the total symbol are, in general, not (yet) available in the manifold setting.

In the present paper we will focus on the validity of what we shall call {\it subelliptic sharp G\r{a}rding inequality}, that is, on the suitable formulation of the sharp G\r{a}rding inequality on compact Lie groups for pseudo-differential operators with symbols belonging to global subelliptic classes.
The aforementioned subelliptic classes, and the corresponding pseudo-differential calculus, are developed in \cite{CardonaRuzhanskyC} by using the sub-Riemannian structure of an arbitrary compact Lie group $G$ and the matrix-valued quantization developed in \cite{Ruz}.

Since any sub-Riemannian structure on $G$ is encoded in terms of a positive sub-Laplacian $\mathcal{L}$ over a compact Lie group $G,$ the global subelliptic H\"ormander classes of symbols in \cite{CardonaRuzhanskyC} were denoted by $S^{m,\,\mathcal{L}}_{\rho,\delta}(G\times \widehat{G}),$ where $m\in \mathbb{R}$ and $0\leq \delta<\rho\leq 1.$   

\subsection{The subelliptic sharp G\r{a}rding inequality}
The statement of our main result, that is of the subelliptic sharp G\r{a}rding inequality on a compact Lie group $G$, is given in Theorem \ref{MainTheorem} below. Here and in the rest of the paper we shall denote by  $H^{s,\,\mathcal{L}}(G)$, for $s\in \mathbb{R}$, the subelliptic Sobolev space of order $s$ associated with a fixed positive sub-Laplacian $\mathcal{L}$, that is, the space defined as the completion of $C^\infty(G)$ with respect to the norm $\Vert u\Vert_{H^{s,\,\mathcal{L}}(G)}:=\Vert (1+\mathcal{L})^{\frac{s}{2}}u\Vert_{L^2(G)}.$
\begin{theorem}[Subelliptic sharp G\r{a}rding inequality]\label{MainTheorem} Let $G$ be a compact Lie group and let $\mathcal{L}=\mathcal{L}_X\footnote{$\mathcal{L}_X:=-X_1^2-X_2^2-\cdots - X_{k}^2.$}$ be the (positive) sub-Laplacian associated with a system  $X=\{X_{i}\}_{i=1}^{k}$ of left-invariant vector fields  satisfying H\"ormander's condition of step $\kappa$ \footnote{which means that the vector fields $X_1,\cdots, X_{k},$ together with their commutators of length at most $\kappa$ span the Lie algebra $\mathfrak{g}$ of $G$ (under the identification $\mathfrak{g}\cong T_{e}G$, with $e\in G$ being the neutral element). If $\mathcal{L}_X$ is a positive Laplacian, we trivially have $\kappa=1.$}. For $0< \rho\leqslant 1$ and  $0\leq \delta<(2\kappa-1)^{-1}\rho$, and for $m\in \mathbb{R}$, let  $$A\equiv a(x,D):C^\infty(G)\rightarrow\mathscr{D}'(G)$$ be a continuous linear operator with global symbol  $a\in {S}^{m,\,\mathcal{L}}_{\rho,\delta}( G\times \widehat{G})$. Then, if $a(x,[\xi])\geq 0$ for all $(x,[\xi])\in G\times \widehat{G}$, there exists a positive constant $C$ such that
\begin{equation}\label{SGIT}
    \mathsf{Re}(Au,u)\geq -C\Vert u\Vert_{H^{\frac{m-\frac 1 \kappa(\rho-(2\kappa-1)\delta)}{2},\,\mathcal{L}}(G)}^2,
\end{equation}
for all $u\in C^{\infty}(G)$, where $(\cdot,\cdot)$ stands for the $L^2$-scalar product.
\end{theorem}
As a consequence of the previous result we obtain in Corollary \ref{Corolario} below the {\it elliptic} sharp G\r{a}rding inequality for operators with symbols in the H\"ormander classes $\mathscr{S}^{m}_{\rho,\delta}(G)$ (corresponding to the case when $\kappa=1$). 

Corollary \ref{Corolario} extends the main result in \cite{RuzhanskyTurunen2011} holding for operators with symbols belonging to the particular class $\mathscr{S}^{m}_{1,0}(G)$.
\begin{corollary}[Elliptic sharp G\r{a}rding Inequality]\label{Corolario}  For $0\leqslant \delta<\rho\leqslant 1,$  let $A\equiv a(x,D):C^\infty(G)\rightarrow\mathscr{D}'(G)$ be a continuous linear operator with symbol  $a\in \mathscr{S}^{m}_{\rho,\delta}( G)$, $m\in \mathbb{R}$. Let us assume that $a(x,[\xi])\geq 0$ for every $(x,[\xi])\in G\times \widehat{G}.$ Then, there exists $C>0$, such that
\begin{equation}
    \mathsf{Re}(Au,u)\geq -C\Vert u\Vert_{H^{\frac{m-(\rho-\delta)}{2}}(G)}^2
\end{equation}
for all $u\in C^{\infty}(G).$
\end{corollary}

Let us briefly discuss some immediate consequences of our main Theorem \ref{MainTheorem}.
\begin{remark}
Notice that when $\mathcal{L}_X=\Delta$ is the Laplacian on the group (that is $X=\{X_i\}_{i=1}^{n}$ is a basis of the Lie algebra $\mathfrak{g}$, or, equivalently, $\kappa=1$), then Theorem \ref{MainTheorem} provides the sharp G\r{arding inequality} for operators with nonnegative symbols in the  classes  $ \mathscr{S}^{m}_{\rho,\delta}( G)$ with $0\leq \delta<\rho\leq 1$ (see Corollary \ref{Corolario}).
This, in particular, shows that our result extends the one in \cite{RuzhanskyTurunen2011} where the {\it elliptic} Sharp G\r{a}rding inequality (namely for the standard global non subelliptic symbols defined in \cite{Ruz}) was proved only for $(\rho,\delta)=(1,0)$. More remarkably, the result applies to all  $\mathscr{S}^m_{\rho,\delta}(G)$ classes of global symbols with $0\leq \delta<\rho\leq 1$, and not only to those where $\rho\geq 1-\delta$ and corresponding to the standard H\"ormander classes.
\end{remark}
\begin{remark}
Observe that the {\it elliptic} Sharp G\r{a}rding inequality for $\mathscr{S}^m_{\rho,\delta}(G)$-classes with $(\rho,\delta)$ in the whole admissible range (see Corollary \ref{Corolario}), is already a new result in the case of the torus $G=\mathbb{T}^n$. As remarked above, the standard local theory allows to consider a restricted range for $\rho$ and $\delta$, therefore our result is much better that the one possibly obtainable trough the local theory, since here the parameters $\rho$ and $\delta$ can be taken in the full range $0\leq \delta<\rho\leq 1$, allowing also the case where  $\rho< 1-\delta.$ 
\end{remark}
As regards the purely subelliptic seeting, the appearence of the parameter $\kappa$ (which is related to the subelliptic order of the fixed sub-Laplacian) in the subelliptic Sobolev norms in \eqref{SGIT}, is dictated by the combination of the noncommutativity of the group and the noncommutativity property of the symbols of sub-Laplacians. Indeed, while the Laplacian is a central operator having matrix-valued global symbol commuting with any other symbol, no sub-Laplacian has the same commutativity properties. Therefore, in order not to restrict our analysis to very particular subclasses of symbols, we combined {\it elliptic} with {\it subelliptic} strategies. 

Note that, as observed by the third author and Fischer in \cite{FisherRuzhansky2013}, the same problem arises in the nilpotent Lie group setting. In fact, due the intrinsic noncommutativity and subellipticity of the setting, in \cite{FisherRuzhansky2013} the sharp G\r{a}rding inequality in the nilpotent setting is announced for very  particular operators, that is, roughly speaking, for those commuting with the fixed sub-Laplacian.

We want to stress that no such commutativity condition is assumed in Theorem \ref{MainTheorem}. However there is a price to pay to work in this general framework, price that is given by a restriction on the classes to which the result applies, namely those such that $0\leq \delta<(2\kappa-1)^{-1}\rho$. Note that the case $(\rho,\delta)=(1,0)$ is covered by our result. 

Let us mention that even imposing a commutativity condition in the same spirit as in \cite{FisherRuzhansky2013}, the use of the global pseudo-differential subelliptic calculus does not lead to the expected result (that is \eqref{SGIT} without the appearence of $\kappa$ in the Sobolev norm), and that the expeceted result for very special classes can most probably be reached via the global functional calculus.

More remarks about the strenght of our result in the purely subelliptic setting are given in Section \ref{FinalRemarks}. There we also show that our subelliptic result does not follow from the elliptic one.

We now conclude this introduction by giving the plan of the paper.
\begin{itemize}
    \item In Section \ref{Sect2} we recall some basic facts about pseudo-differential operators on compact Lie groups, we  recall the subelliptic global symbol classes in \cite{CardonaRuzhanskyC} that are the object of our analysis. At the end of the section an analysis of amplitude subelliptic operators is consistently developed.
    \item In Section \ref{Sect3} we focus on the proof of the main theorem about the subelliptic sharp G\r{a}rding inequality.
    \item Finally, Section \ref{FinalRemarks} is devoted to some remarks about our result in the purely subelliptic setting.
\end{itemize}

\section{Sub-Laplacians and pseudo-differential operators on compact Lie groups}\label{Sect2} 

\subsection{Pseudo-differential operators via localisations}\label{S2.1}
 Pseudo-differential operators on compact manifolds, and consequently on compact Lie groups, can be defined by using local coordinate charts (see H\"ormander \cite{Hormander1985III} and also M. Taylor \cite{Taylorbook1981} for a good introductory background on the subject).
 
Let us briefly  introduce these  classes starting with the definition in the Euclidean setting. Let $U$ be an open  subset of $\mathbb{R}^n.$ We  say that  the ``symbol" $a\in C^{\infty}(U\times \mathbb{R}^n, \mathbb{C})$ belongs to the H\"ormander class of order $m$ and of $(\rho,\delta)$-type, $S^m_{\rho,\delta}(U\times \mathbb{R}^n),$ $0\leqslant \rho,\delta\leqslant 1,$ if for every compact subset $K\subset U$ and for all $\alpha,\beta\in \mathbb{N}_0^n$, the symbol inequalities
\begin{equation*}
  |\partial_{x}^\beta\partial_{\xi}^\alpha a(x,\xi)|\leqslant C_{\alpha,\beta,K}(1+|\xi|)^{m-\rho|\alpha|+\delta|\beta|},
\end{equation*} hold true uniformly in $x\in K$ for all  $\xi\in \mathbb{R}^n.$ Then, a continuous linear operator $A:C^\infty_0(U) \rightarrow C^\infty(U)$ 
is a pseudo-differential operator of order $m$ of  $(\rho,\delta)$-type, if there exists
a symbol $a\in S^m_{\rho,\delta}(U\times \mathbb{R}^n)$ such that
\begin{equation*}
    Af(x)=\int\limits_{\mathbb{R}^n}e^{2\pi i x\cdot \xi}a(x,\xi)(\mathscr{F}_{\mathbb{R}^n}{f})(\xi)d\xi,
\end{equation*} for all $f\in C^\infty_0(U),$ where
$$
    (\mathscr{F}_{\mathbb{R}^n}{f})(\xi):=\int\limits_Ue^{-i2\pi x\cdot \xi}f(x)dx
$$ is the  Euclidean Fourier transform of $f$ at $\xi\in \mathbb{R}^n.$ 

Once the definition of H\"ormander classes on open subsets of $\mathbb{R}^n$ is established, it can be extended to smooth manifolds as follows.  Given a $C^\infty$-manifold $M,$ a linear continuous operator $A:C^\infty_0(M)\rightarrow C^\infty(M) $ is a pseudo-differential operator of order $m$ of $(\rho,\delta)$-type, with $ \rho\geqslant   1-\delta, $ and $0\leq \delta<\rho\leq 1,$  if for every local  coordinate patch $\omega: M_{\omega}\subset M\rightarrow U_{\omega}\subset \mathbb{R}^n,$
and for every $\phi,\psi\in C^\infty_0(U_\omega),$ the operator
\begin{equation*}
    Tu:=\psi(\omega^{-1})^*A\omega^{*}(\phi u),\,\,u\in C^\infty(U_\omega),\footnote{As usually, $\omega^{*}$ and $(\omega^{-1})^*$ are the pullbacks, induced by the maps $\omega$ and $\omega^{-1}$ respectively.}
\end{equation*} is a standard pseudo-differential operator with symbol $a_T\in S^m_{\rho,\delta}(U_\omega\times \mathbb{R}^n).$ In this case we write $A\in \Psi^m_{\rho,\delta}(M,\textnormal{loc}).$

\subsection{The positive sub-Laplacian and pseudo-differential operators via global symbols} Let $G$ be a compact Lie group with Lie algebra $\mathfrak{g}\simeq T_{e_G}G$, where $e_{G}$ is the neutral element of $G$, and let  
$$X=\{X_1,\cdots,X_{k} \}\subset \mathfrak{g}$$ 
be a system of $C^\infty$-vector fields. For all $I=(i_1,\cdots,i_\omega)\in \{1,2,\cdots,k\}^{\omega}$ of length $\omega\geqslant   1$, we denote by $$X_{I}:=[X_{i_1},[X_{i_2},\cdots [X_{i_{\omega-1}},X_{i_\omega}]\cdots]]$$
a commutator of length $\omega$, where $X_{I}:=X_{i}$ when $\omega=1$ and $I=(i)$. The system $X$ is said to satisfy H\"ormander's condition of step (or order) $\kappa$ if $\mathfrak{g}=\mathrm{span}\{X_I: |I|\leq \kappa\}$, that is, in other words, the vector fields $X_j$, $j=1,\ldots, k$, together with their commutator up to length $\kappa$, generate the whole Lie algebra $\mathfrak{g}$. 

Note that we are assuming that there is no subsystem $Y=\{Y_1,\cdots,Y_{\ell} \}\subset X$, $\ell< k$, of smooth vector fields such that $\mathfrak{g}=\mathrm{span}\{Y_I: |I|\leq \kappa\}$.
In this case we call $X$ a system of H\"ormander's vector fields.

Given a system $X=\{X_1,\cdots,X_{k}\}$ of H\"ormander's vector fields, then the operator defined as
\begin{equation*}
    \mathcal{L}\equiv \mathcal{L}_{X}:=-(X_{1}^2+\cdots +X_{k}^2),
\end{equation*} 
is a hypoelliptic operator by H\"ormander theorem on sums of the  squares of vector fields (see H\"ormander \cite{Hormander1967}). In particular the operator $\mathcal{L}$ is also subelliptic, and it is called the subelliptic Laplacian associated with the system $X$, or simply sub-Laplacian. It is clear from the definition that one can define different sub-Laplacians by using  different systems of  H\"ormander's vector fields (and that satisfy H\"ormander condition of different step). 

We will not treat other aspects of the analysis of sub-Laplacians here, we refer the interested reader to Agrachev et al. \cite{Agrachev2008}, Bismut \cite{Bismut2008}, Domokos et al. \cite{Domokos}, and to  the fundamental book of Montgomery \cite{Montgomery}. For some applications of H\"ormander's vector fields we refer to the book of Bramanti \cite{Bramanti}.  

Let us now introduce the Hausdorff dimension associated with the sub-Laplacian $\mathcal{L}$. For all $x\in G,$ let $H_{x}^\omega G$ be the linear subspace of the tangent space $T_xG$ generated by the $X_i$'s and by all the Lie brackets $$ [X_{j_1},X_{j_2}],[X_{j_1},[X_{j_2},X_{j_3}]],\cdots, [X_{j_1},[X_{j_2}, [X_{j_3},\cdots, X_{j_\omega}] ] ],$$ with $\omega\leqslant \kappa.$ Then clearly H\"ormander condition can be stated as $H_{x}^\kappa G=T_xG$ for all $x\in G$, where the following inclusions hold
\begin{equation*}
H_{x}^1G\subset H_{x}^2G \subset H_{x}^3G\subset \cdots \subset H_{x}^{\kappa-1}G\subset H_{x}^\kappa G= T_xG,\,\,x\in G.
\end{equation*} Note that the dimension of every $H_x^\omega G$ is constant in $x\in G$, so we set $\dim H^\omega G:=\dim H_{x}^\omega G,$ for all $x\in G$, and have that the Hausdorff dimension can be defined as (see e.g. \cite[p. 6]{HK16}),
\begin{equation}\label{Hausdorff-dimension}
    Q:=\dim(H^1G)+\sum_{i=1}^{\kappa-1} (i+1)(\dim H^{i+1}G-\dim H^{i}G ).
\end{equation}

As already mentioned in the introduction, we will make use of the quantization process developed by the third author and V. Turunen in \cite{Ruz}. We briefly recall below how this global quantization is defined. 

Let  $A$ be a continuous linear operator from $C^\infty(G)$ into $\mathscr{D}'(G),$ and let  $\widehat{G}$ be  the algebraic unitary dual of $G.$ Then, there exists a function \begin{equation}\label{symbol}a:G\times \widehat{G}\rightarrow \cup_{\ell\in \mathbb{N}} \mathbb{C}^{\ell\times \ell},\end{equation}  that we call the symbol of $A,$ such that $a(x,\xi):=a(x,[\xi])\in \mathbb{C}^{d_\xi\times d_\xi}$ for every equivalence class $[\xi]\in \widehat{G},$ where $\xi:G\rightarrow \textnormal{Hom}(H_{\xi}),$ $H_{\xi}\cong \mathbb{C}^{d_\xi},$ and such that
\begin{equation}\label{RuzhanskyTurunenQuanti}
    Af(x)=\sum_{[\xi]\in \widehat{G}}d_\xi{\textnormal{Tr}}[\xi(x)a(x,\xi)\widehat{f}(\xi)],\,\,\forall f\in C^\infty(G).
\end{equation}Note that we have denoted by
\begin{equation*}
    \widehat{f}(\xi)\equiv (\mathscr{F}f)(\xi):=\int\limits_{G}f(x)\xi(x)^*dx\in  \mathbb{C}^{d_\xi\times d_\xi},\,\,\,[\xi]\in \widehat{G},
\end{equation*} the Fourier transform of $f$ at $\xi\cong(\xi_{ij})_{i,j=1}^{d_\xi},$ where the matrix representation of $\xi$ is induced by an orthonormal basis of the representation space $H_{\xi}.$
The function $a$ in \eqref{symbol} satisfying \eqref{RuzhanskyTurunenQuanti} is unique, and satisfies the identity
\begin{equation*}
    a(x,\xi)=\xi(x)^{*}(A\xi)(x),\,\, A\xi:=(A\xi_{ij})_{i,j=1}^{d_\xi},\,\,\,[\xi]\in \widehat{G}.
\end{equation*}
Note that the previous identity is well defined. Indeed, it is well known that the functions $\xi_{ij}$, which are of $C^\infty$-class, are the eigenfunctions of the positive Laplace operator $\mathcal{L}_G$, that is $\mathcal{L}_G\xi_{ij}=\lambda_{[\xi]}\xi_{ij}$ for some non-negative real number $\lambda_{[\xi]}\geq 0$
depending  only of the equivalence class $[\xi]$ and not on the representation $\xi$.

In general, we refer to the function $a$ as the (global or full) {\it{symbol}} of the operator $A,$ and we will use the notation $A=\textnormal{Op}(a)$ to indicate that $a:=\sigma_A$ is the symbol associated with the operator $A.$

In order to classify symbols in the H\"ormander classes, in  \cite{Ruz} the authors defined the notion of {\it{difference operators}}, which endows $\widehat{G}$ with a difference structure.  Following  \cite{RuzhanskyWirth2015}, a difference operator $Q_\xi: \mathscr{D}'(\widehat{G})\rightarrow \mathscr{D}'(\widehat{G})$ of order $k$  is defined as
\begin{equation}\label{taylordifferences}
    Q_\xi\widehat{f}(\xi)=\widehat{qf}(\xi),\,[\xi]\in \widehat{G},
\end{equation}  for some function $q$ vanishing of order $k$ at the neutral element $e=e_G.$ We will denote by $\textnormal{diff}^k(\widehat{G})$  the class of all difference operators of order $k.$ For a  fixed smooth function $q,$ the associated difference operator will be denoted by $\Delta_q\equiv Q_\xi.$ A system  of difference operators (see e.g. \cite{RuzhanskyWirth2015})
\begin{equation*}
  \Delta_{\xi}^\alpha:=\Delta_{q_{(1)}}^{\alpha_1}\cdots   \Delta_{q_{(i)}}^{\alpha_{i}},\,\,\alpha=(\alpha_j)_{1\leqslant j\leqslant i}, 
\end{equation*}
with $i\geq n$, is called   admissible  if
\begin{equation*}
    \textnormal{rank}\{\nabla q_{(j)}(e):1\leqslant j\leqslant i \}=\textnormal{dim}(G), \textnormal{   and   }\Delta_{q_{(j)}}\in \textnormal{diff}^{1}(\widehat{G}).
\end{equation*}
An admissible collection is said to be strongly admissible if, additionally, 
\begin{equation*}
    \bigcap_{j=1}^{i}\{x\in G: q_{(j)}(x)=0\}=\{e_G\}.
\end{equation*}

\begin{remark}\label{remarkD}
Matrix components of unitary representations induce difference operators. Indeed, if $\xi_{1},\xi_2,\cdots, \xi_{k},$ are  fixed irreducible and unitary  representation of $G$, which not necessarily belong to the same equivalence class, then each coefficient of the matrix
\begin{equation}
 \xi_{\ell}(g)-I_{d_{\xi_{\ell}}}=[\xi_{\ell}(g)_{ij}-\delta_{ij}]_{i,j=1}^{d_{\xi_\ell}},\, \quad g\in G, \,\,1\leq \ell\leq k,
\end{equation} 
that is each function 
$q^{\ell}_{ij}(g):=\xi_{\ell}(g)_{ij}-\delta_{ij}$, $ g\in G,$ defines a difference operator
\begin{equation}\label{Difference:op:rep}
    \mathbb{D}_{\xi_\ell,i,j}:=\mathscr{F}(\xi_{\ell}(g)_{ij}-\delta_{ij})\mathscr{F}^{-1}.
\end{equation}
We can fix $k\geq \mathrm{dim}(G)$ of these representations in such a way that the corresponding  family of difference operators is admissible, that is, 
\begin{equation*}
    \textnormal{rank}\{\nabla q^{\ell}_{i,j}(e):1\leqslant \ell\leqslant k \}=\textnormal{dim}(G).
\end{equation*}
To define higher order difference operators of this kind, let us fix a unitary irreducible representation $\xi_\ell$.
Since the representation is fixed we omit the index $\ell$ of the representations $\xi_\ell$ in the notation that will follow.
Then, for any given multi-index $\alpha\in \mathbb{N}_0^{d_{\xi_\ell}^2}$, with 
$|\alpha|=\sum_{i,j=1}^{d_{\xi_\ell}}\alpha_{i,j}$, we write
$$\mathbb{D}^{\alpha}:=\mathbb{D}_{1,1}^{\alpha_{11}}\cdots \mathbb{D}^{\alpha_{d_{\xi_\ell},d_{\xi_\ell}}}_{d_{\xi_\ell}d_{\xi_\ell}}
$$ 
for a difference operator of order $|\alpha|$.
\end{remark}
The difference operators endow the unitary dual $\widehat{G}$ with a difference structure. For difference operators of the previous form, the  following finite Leibniz-like formula holds true (see  \cite{RuzhanskyTurunenWirth2014} for details). Note that below we are still assuming that the representation $\xi_\ell$ is fixed.
\begin{proposition}[Leibniz rule for difference operators]\label{Leibnizrule} Let $G$ be a compact Lie group and let $\mathbb{D}^{\alpha},$ $\alpha\in \mathbb{N}^{d_{\xi_\ell}}_0,$ be the family of difference operators defined in  \eqref{Difference:op:rep}. Then, the following Leibniz rule
\begin{align*}
   ( \mathbb{D}^{\alpha}a_{1}a_{2})(x_0,\xi)(x_0,\xi) )=\sum_{ |\gamma|,|\varepsilon|\leqslant |\alpha|\leqslant |\gamma|+|\varepsilon| }C_{\varepsilon,\gamma}(\mathbb{D}^{\gamma}a_{1})(x_0,\xi) (\mathbb{D}^{\varepsilon}a_{2})(x_0,\xi), \quad x_{0}\in G,
\end{align*}holds  for all $a_{1},a_{2}\in C^{\infty}(G)\times \mathscr{S}'(\widehat{G})$, where the summation is taken over all $\varepsilon, \gamma$ such that $|\varepsilon|,|\delta|\leq |\alpha|\leq |\gamma|+|\varepsilon|$. 
\end{proposition}
Note that for different kind of difference operators, namely for those given by compositions of difference operators of higher order associated with different representations, a Leibniz-like formula still holds true by iteration. For more details about difference operators and Leibniz-like formulas for admissible collections see also Corollary 5.13 in \cite{Fischer2015}.
\medskip

We are now going to introduce the global H\"ormander classes of symbols defined in \cite{Ruz}.
First let us recall that every left-invariant vector field  $Y\in\mathfrak{g}$ can be identified with the first order  differential operator $\partial_{Y}:C^\infty(G)\rightarrow \mathscr{D}'(G)$  given by
 \begin{equation*}
   \partial_{Y}f(x)=  (Y_{x}f)(x)=\frac{d}{dt}f(x\exp(tY) )|_{t=0}.
 \end{equation*}If $\{X_1,\cdots, X_n\}$ is a basis of the Lie algebra $\mathfrak{g},$ then we will use the standard multi-index notation
 $$ \partial_{X}^{\alpha}=X_{x}^{\alpha}=\partial_{X_1}^{\alpha_1}\cdots \partial_{X_n}^{\alpha_n},     $$
 for a canonical left-invariant differential operator of order $|\alpha|.$

By using this property, together with the following notation for the so-called  elliptic weight $$\langle\xi \rangle:=(1+\lambda_{[\xi]})^{1/2},\,\,[\xi]\in \widehat{G},$$ we can finally give the definition of global symbol classes.
\begin{definition}Let $G$ be a compact Lie group and let $0\leqslant \delta,\rho\leqslant 1.$ Let $$\sigma:G\times \widehat{G}\rightarrow \bigcup_{[\xi]\in \widehat{G}}\mathbb{C}^{d_\xi\times d_\xi},$$ be a matrix-valued function such that for any $[\xi]\in \widehat{G},$ $\sigma(\cdot,[\xi])$ is of $C^\infty$-class, and such that, for any given $x\in G$ there is a distribution $k_{x}\in \mathscr{D}'(G),$ smooth in $x,$ satisfying $\sigma(x,\xi)=\widehat{k}_{x}(\xi),$ $[\xi]\in \widehat{G}$. We say that $\sigma \in \mathscr{S}^{m}_{\rho,\delta}(G)$ if the following symbol inequalities 
\begin{equation}\label{HormanderSymbolMatrix}
   \Vert \partial_{X}^\beta \Delta_\xi^{\gamma} \sigma_A(x,\xi)\Vert_{\textnormal{op}}\leqslant C_{\alpha,\beta}
    \langle \xi \rangle^{m-\rho|\gamma|+\delta|\beta|},
\end{equation} are satisfied for all $\beta$ and  $\gamma $ multi-indices and for all $(x,[\xi])\in G\times \widehat{G}.$ For $\sigma_A\in \mathscr{S}^m_{\rho,\delta}(G)$ we will write $A\in\Psi^m_{\rho,\delta}(G)\equiv\textnormal{Op}(\mathscr{S}^m_{\rho,\delta}(G)).$
\end{definition}
The global H\"ormander classes on compact Lie groups can be used to describe the H\"ormander classes defined by local coordinate systems. We present the corresponding statement as follows. 
\begin{theorem}[Equivalence of classes, \cite{Fischer2015,Ruz,RuzhanskyTurunenWirth2014}] Let $A:C^{\infty}(G)\rightarrow\mathscr{D}'(G)$ be a continuous linear operator and let $0\leq \delta<\rho\leq 1,$ with $\rho\geq 1-\delta.$ Then, $A\in \Psi^m_{\rho,\delta}(G,\textnormal{loc}),$ if and only if $\sigma_A\in \mathscr{S}^m_{\rho,\delta}(G).$ Consequently,
\begin{equation}\label{EQequivalence}
   \textnormal{Op}(\mathscr{S}^m_{\rho,\delta}(G))= \Psi^m_{\rho,\delta}(G,\textnormal{loc}),\,\,\,0\leqslant \delta<\rho\leqslant 1,\,\rho\geqslant   1-\delta.
\end{equation}
\end{theorem}

\subsection{Odd and even functions on compact Lie groups} 
Since we will use some properties of odd and even functions on $G$ in this paper, for completeness we recall these definitions below.
\begin{definition}
   On a group $G,$ a function $f:G\rightarrow \mathbb{C}$ is
   \begin{itemize}
       \item  even, if $f(x^{-1})=f(x),$ for every $x\in G$;
       \item  central, if $f(xy)=f(yx),$ for every $x,y\in G$;
       \item odd, if $f(x^{-1})=-f(x),$ for every $x\in G$.
   \end{itemize}
\end{definition}
Next, we summarize the action of vector fields on even and odd functions (see Proposition 3.11 of \cite{RuzhanskyTurunen2011}). 
\begin{proposition}
Let $G$ be a Lie group and $f\in C^\infty(G).$ Let $X_{i},$ $1\leq i\leq s$, be an arbitrary system of vector fields in $\mathfrak{g}.$ Then
\begin{equation}
    \partial_{X_1}\cdots \partial_{X_s}f(x^{-1})=\pm \partial_{X_s}\cdots \partial_{X_1}f(x),
\end{equation}
and
\begin{equation*}
    \partial_{X_1}\cdots \partial_{X_s}f(x^{-1})=(-1)^{s+1} \partial_{X_s}\cdots \partial_{X_1}f(x)
\end{equation*}
if $f$ is an odd central function.

\end{proposition}

\subsection{Subelliptic H\"ormander classes on compact Lie groups} 
 In order to define the subelliptic H\"ormander calculus, we will use a suitable basis of the Lie algebra arising from Taylor expansions.  We explain the choice of this basis by means of the following lemma (see Lemma 7.4 in \cite{Fischer2015}).

 \begin{lemma}\label{Taylorseries} Let $G$ be a compact Lie group of dimension $n.$ Let $\mathfrak{D}=\{\Delta_{q_{(j)}}\}_{1\leqslant j\leqslant n}$ be a strongly admissible admissible collection of difference operators, that is
\begin{equation*}
    \textnormal{rank}\{\nabla q_{(j)}(e):1\leqslant j\leqslant n \}=n, \,\,\,\bigcap_{j=1}^{n}\{x\in G: q_{(j)}(x)=0\}=\{e_G\}.
\end{equation*}
Then there exists a basis $X_{\mathfrak{D}}=\{X_{1,\mathfrak{D}},\cdots ,X_{n,\mathfrak{D}}\}$ of $\mathfrak{g}$ such that $$X_{j,\mathfrak{D}}q_{(k)}(\cdot^{-1})(e_G)=\delta_{jk}.
$$
Moreover, by using the multi-index notation $$\partial_{X}^{(\beta)}=\partial_{X_{1,\mathfrak{D}}}^{\beta_1}\cdots \partial_{X_{n,\mathfrak{D}}}^{\beta_n}, $$ for any $\beta\in\mathbb{N}_0^n,$
where $$\partial_{X_{i,\mathfrak{D}}}f(x)=  \frac{d}{dt}f(x\exp(tX_{i,\mathfrak{D}}) )|_{t=0},\,\,f\in C^{\infty}(G),$$ and denoting by
\begin{equation*}
    R_{x,N}^{f}(y)=f(xy)-\sum_{|\alpha|<N}q_{(1)}^{\alpha_1}(y^{-1})\cdots q_{(n)}^{\alpha_n}(y^{-1})\partial_{X}^{(\alpha)}f(x)
\end{equation*} 
the Taylor remainder, we have that 
\begin{equation*}
    | R_{x,N}^{f}(y)|\leqslant C|y|^{N}\max_{|\alpha|\leqslant N}\Vert \partial_{X}^{(\alpha)}f\Vert_{L^\infty(G)},
\end{equation*}
where the constant $C>0$ is dependent on $N,$ $G$ and $\mathfrak{D}$ (but not on $f\in C^\infty(G)).$ In addition we have that $\partial_{X}^{(\beta)}|_{x_1=x}R_{x_1,N}^{f}=R_{x,N}^{\partial_{X}^{(\beta)}f}$, and 
\begin{equation*}
    | \partial_{X}^{(\beta)}|_{y_1=y}R_{x,N}^{f}(y_1)|\leqslant C|y|^{N-|\beta|}\max_{|\alpha|\leqslant N-|\beta|}\Vert \partial_{X}^{(\alpha+\beta)}f\Vert_{L^\infty(G)},
\end{equation*}provided that $|\beta|\leqslant N.$
 \end{lemma}

Using the notation above, and denoting by $\Delta_{\xi}^\alpha:=\Delta_{q_{(1)}}^{\alpha_1}\cdots   \Delta_{q_{(n)}}^{\alpha_{n}},$ we can introduce the subelliptic H\"ormander class of symbols of order $m\in \mathbb{R}$ of type $(\rho,\delta)$.

\begin{definition}[Subelliptic H\"ormander classes]\label{contracted''}
   Let $G$ be a compact Lie group and let $0\leqslant \delta,\rho\leqslant 1.$ Let us consider a sub-Laplacian $\mathcal{L}=-(X_1^2+\cdots +X_{k}^2)$ on $G,$ where the system of vector fields $X=\{X_i\}_{i=1}^{k}$ satisfies the H\"ormander condition of step $\kappa$.  We say that $\sigma \in {S}^{m,\,\mathcal{L}}_{\rho,\delta}(G\times \widehat{G})$ if 
   \begin{equation}\label{InIC}
      p_{\alpha,\beta,\rho,\delta,m,\textnormal{left}}(\sigma)':= \sup_{(x,[\xi])\in G\times \widehat{G} }\Vert \widehat{ \mathcal{M}}(\xi)^{(\rho|\alpha|-\delta|\beta|-m)}\partial_{X}^{(\beta)} \Delta_{\xi}^{\alpha}\sigma(x,\xi)\Vert_{\textnormal{op}} <\infty,
   \end{equation}
   \begin{equation}\label{InIIC}
      p_{\alpha,\beta,\rho,\delta,m,\textnormal{right}}(\sigma)':= \sup_{(x,[\xi])\in G\times \widehat{G} }\Vert (\partial_{X}^{(\beta)} \Delta_{\xi}^{\alpha} \sigma(x,\xi) ) \widehat{ \mathcal{M}}(\xi)^{(\rho|\alpha|-\delta|\beta|-m)}\Vert_{\textnormal{op}} <\infty,
   \end{equation} holds true for all $\alpha, \beta\in \mathbb{N}^n_0$.
  \end{definition}
By following the usual nomenclature, we  define:
\begin{equation*}
    \textnormal{Op}({S}^{m,\,\mathcal{L}}_{\rho,\delta}(G\times \widehat{G})):=\{A:C^{\infty}(G)\rightarrow \mathscr{D}'(G):\sigma_A\equiv\widehat{A}(x,\xi)\in {S}^{m,\,\mathcal{L}}_{\rho,\delta}(G\times \widehat{G}) \},
\end{equation*} with
\begin{equation*}
    Af(x)=\sum_{[\xi]\in \widehat{G}}d_\xi \textnormal{{Tr}}(\xi(x)\widehat{A}(x,\xi)\widehat{f}(\xi)),\,\,\,f\in C^\infty(G),\,x\in G.  
\end{equation*}
We will use the notation    $\widehat{ \mathcal{M}}$ for  the matrix-valued symbol of the operator $\mathcal{M}:=(1+\mathcal{L})^{\frac{1}{2}},$ and,  for every $[\xi]\in \widehat{G}$ and $s\in \mathbb{R},$ we define
   \begin{equation*}
       \widehat{ \mathcal{M}}(\xi)^{s}:=\textnormal{diag}[(1+\nu_{ii}(\xi)^2)^{\frac{s}{2}}]_{1\leqslant i\leqslant d_\xi},
   \end{equation*} where $\widehat{\mathcal{L}}(\xi)=:\textnormal{diag}[\nu_{ii}(\xi)^2]_{1\leqslant i\leqslant d_\xi}$ is the symbol of the sub-Laplacian $\mathcal{L}$ at $[\xi].$

\begin{definition}[Subelliptic amplitudes] A function $a:G\times G\times \widehat{G}\rightarrow \cup_{[\xi]\in \widehat{G}}\mathbb{C}^{d_\xi\times d_\xi}$ is an amplitude symbol if for every $[\xi]\in \widehat{G},$ $a(\cdot,\cdot,[\xi])$ is smooth. In addition, $a$ belongs to the subelliptic amplitude class of order $m$ and type $(\rho,\delta)$, $\mathcal{A}^{m,\,\mathcal{L}}_{\rho,\delta}(G\times G\times \widehat{G})$ if
   \begin{equation}\label{InIC2}
     \sup_{(x,y,[\xi])\in G\times G\times  \widehat{G} }\Vert \widehat{ \mathcal{M}}(\xi)^{(\rho|\alpha|-\delta(|\beta|+|\gamma|)-m)}\partial_{X}^{(\beta)}\partial_{Y}^{(\gamma)} \Delta_{\xi}^{\alpha}a(x,y,\xi)\Vert_{\textnormal{op}} <\infty,
   \end{equation}
   and 
   \begin{equation}\label{InIIC2}
      \sup_{(x,y,[\xi])\in G\times G\times  \widehat{G} }\Vert (\partial_{X}^{(\beta)}\partial_{Y}^{(\gamma)} \Delta_{\xi}^{\alpha} a(x,y,\xi) ) \widehat{ \mathcal{M}}(\xi)^{(\rho|\alpha|-\delta(|\beta|+|\gamma|)-m)}\Vert_{\textnormal{op}} <\infty.
   \end{equation}
  The amplitude operator associated with an amplitude $a\in \mathcal{A}^{m,\,\mathcal{L}}_{\rho,\delta}(G\times G\times \widehat{G})$ is defined via
\begin{equation*}
    Af(x)\equiv \textnormal{AOp}(a)f(x):=\sum_{[\xi]\in \widehat{G}}d_\xi \textnormal{{Tr}}\left(\xi(x)\int\limits_{G} a(x,y,\xi)\xi(y)^{*}f(y)dy\right),  
\end{equation*}where $f\in C^\infty(G).$
\end{definition}
The decay properties of subelliptic symbols are summarized in the following lemma (see \cite[Chapter 4]{CardonaRuzhanskyC}).

\begin{lemma}\label{lemadecaying1}
Let $G$ be a compact Lie group and  let $0\leqslant \delta,\rho\leqslant 1.$ If $a\in {S}^{m,\,\mathcal{L}}_{\rho,\delta}(G\times \widehat{G}),$ then for every $\alpha,\beta\in \mathbb{N}_0^n,$ there exists $C_{\alpha,\beta}>0$ satisfying the estimates
\begin{equation*}
    \Vert \partial_{X}^{(\beta)} \Delta_{\xi}^{\alpha}a(x,\xi)\Vert_{\textnormal{op}}\leqslant C_{\alpha,\beta}\sup_{1\leqslant i\leqslant d_\xi}(1+\nu_{ii}(\xi)^2)^{\frac{m-\rho|\alpha|+\delta|\beta|}{2 }},
\end{equation*}uniformly in $(x,[\xi])\in G\times \widehat{G}.$ 
\end{lemma}

In the next theorem we describe some fundamental properties of the subelliptic calculus \cite{CardonaRuzhanskyC}.
\begin{theorem}\label{calculus} Let $0\leqslant \delta<\rho\leqslant 1,$ and let  $\Psi^{m,\,\mathcal{L}}_{\rho,\delta}:=\textnormal{Op}({S}^{m,\,\mathcal{L}}_{\rho,\delta}(G\times \widehat{G})),$ for every $m\in \mathbb{R}.$ Then,
\begin{itemize}
    \item [-] The mapping $A\mapsto A^{*}:\Psi^{m,\,\mathcal{L}}_{\rho,\delta}\rightarrow \Psi^{m,\,\mathcal{L}}_{\rho,\delta}$ is a continuous linear mapping between Fr\'echet spaces and  the  symbol of $A^*,$ $\sigma_{A^*}(x,\xi)$ satisfies the asymptotic expansion,
 \begin{equation*}
    \widehat{A^{*}}(x,\xi)\sim \sum_{|\alpha|= 0}^\infty\Delta_{\xi}^\alpha\partial_{X}^{(\alpha)} (\widehat{A}(x,\xi)^{*}).
 \end{equation*} This means that, for every $N\in \mathbb{N},$ and for all $\ell\in \mathbb{N},$
\begin{equation*}
   \Small{ \Delta_{\xi}^{\alpha_\ell}\partial_{X}^{(\beta)}\left(\widehat{A^{*}}(x,\xi)-\sum_{|\alpha|\leqslant N}\Delta_{\xi}^\alpha\partial_{X}^{(\alpha)} (\widehat{A}(x,\xi)^{*}) \right)\in {S}^{m-(\rho-\delta)(N+1)-\rho\ell+\delta|\beta|,\,\mathcal{L}}_{\rho,\delta}(G\times\widehat{G}) },
\end{equation*} where $|\alpha_\ell|=\ell.$
\item [-] The mapping $(A_1,A_2)\mapsto A_1\circ A_2: \Psi^{m_1,\,\mathcal{L}}_{\rho,\delta}\times \Psi^{m_2,\,\mathcal{L}}_{\rho,\delta}\rightarrow \Psi^{m_3,\,\mathcal{L}}_{\rho,\delta}$ is a continuous bilinear mapping between Fr\'echet spaces, and the symbol of $A=A_{1}\circ A_2$ satisfies is given by the asymptotic formula
\begin{equation*}
    \sigma_A(x,\xi)\sim \sum_{|\alpha|= 0}^\infty(\Delta_{\xi}^\alpha\widehat{A}_{1}(x,\xi))(\partial_{X}^{(\alpha)} \widehat{A}_2(x,\xi)),
\end{equation*}which, in particular, means that, for every $N\in \mathbb{N},$ and for all $\ell \in\mathbb{N},$
\begin{align*}
    &\Delta_{\xi}^{\alpha_\ell}\partial_{X}^{(\beta)}\left(\sigma_A(x,\xi)-\sum_{|\alpha|\leqslant N}  (\Delta_{\xi}^\alpha\widehat{A}_{1}(x,\xi))(\partial_{X}^{(\alpha)} \widehat{A}_2(x,\xi))  \right)\\
    &\hspace{2cm}\in {S}^{m_1+m_2-(\rho-\delta)(N+1)-\rho\ell+\delta|\beta|,\,\mathcal{L}}_{\rho,\delta}(G\times \widehat{G}),
\end{align*}for all  $\alpha_\ell \in \mathbb{N}_0^n$ with $|\alpha_\ell|=\ell.$
\item [-] For  $0\leqslant \delta< \rho\leqslant    1,$  (or for $0\leq \delta\leq \rho\leq 1,$  $\delta<1/\kappa$) let us consider a continuous linear operator $A:C^\infty(G)\rightarrow\mathscr{D}'(G)$ with symbol  $\sigma\in {S}^{0,\,\mathcal{L}}_{\rho,\delta}(G\times \widehat{G})$. Then $A$ extends to a bounded operator from $L^2(G)$ to  $L^2(G).$ 
\end{itemize}
\end{theorem}
\begin{remark}\label{SerenaRemark}
Let us remark that for all $m>0$, we have  $$\mathscr{S}^{\frac{m}{\kappa}}_{\rho,\frac{\delta}{\kappa}}(G)\subset S^{m,\,\mathcal{L}}_{\rho,\delta}(G\times \widehat{G}).$$
To show this property we will make use of the estimate
\begin{equation}\label{GarettoRuzhanskyIneq}
  \langle \xi\rangle^{\frac{1}{\kappa}}\lesssim  (1+\nu_{ii}(\xi)^2)^{\frac{1}{2}}\lesssim \langle \xi\rangle
\end{equation}  proved in Proposition 3.1 of \cite{GarettoRuzhansky2015}.
Indeed, for $\sigma \in \mathscr{S}^{\frac{m}{\kappa}}_{\rho,\frac{\delta}{\kappa}}(G),$ we have that
$$  \|\partial_{X}^{(\beta)}\Delta_\xi^\alpha \sigma(x,\xi)  \|_{\textnormal{op}}\lesssim \langle \xi \rangle^{\frac{m}{\kappa}-\rho|\alpha|+\frac{\delta}{\kappa}|\beta|},  $$ and, consequently, we have
\begin{align*}
     \|\widehat{M}(\xi)^{-m+\rho|\alpha|-\delta|\beta|}\partial_{X}^{(\beta)}\Delta_\xi^\alpha \sigma(x,\xi)  \|_{\textnormal{op}}\leq \sup_{1\leq i\leq d_\xi}\langle \nu_{ii}(\xi)\rangle^{-m+\rho|\alpha|-\delta|\beta|}\langle \xi \rangle^{\frac{m}{\kappa}-\rho|\alpha|+\frac{\delta}{\kappa}|\beta|}.
\end{align*}From the previous inequality, for $-m+\rho|\alpha|+\delta|\beta|<0,$ we obtain  
\begin{align*}
    \langle \nu_{ii}(\xi)\rangle^{-m+\rho|\alpha|-\delta|\beta|}\langle \xi \rangle^{\frac{m}{\kappa}-\rho|\alpha|+\frac{\delta}{\kappa}|\beta|}&\lesssim \langle \xi\rangle^{\frac{-m+\rho|\alpha|-\delta|\beta|}{\kappa}}\langle \xi \rangle^{\frac{m}{\kappa}-\rho|\alpha|+\frac{\delta}{\kappa}|\beta|}\lesssim 1,
\end{align*} while for $-m-\rho|\alpha|+\delta|\beta|\geq 0,$ we have
\begin{align*}
    \langle \nu_{ii}(\xi)\rangle^{-m+\rho|\alpha|-\delta|\beta|}\langle \xi \rangle^{\frac{m}{\kappa}-\rho|\alpha|+\frac{\delta}{\kappa}|\beta|}&\lesssim \langle \xi\rangle^{-m+\rho|\alpha|-\delta|\beta|}\langle \xi \rangle^{\frac{m}{\kappa}-\rho|\alpha|+\frac{\delta}{\kappa}|\beta|}\lesssim 1,
\end{align*}proving that $\sigma\in S^{m,\,\mathcal{L}}_{\rho,\delta}(G\times \widehat{G}).$
\end{remark}

\begin{remark}
The last assertion in Theorem \ref{calculus} remains valid if we consider $0\leq \delta\leq \rho\leq 1,$ $\delta<1/\kappa.$ This is  the subelliptic Calder\'on-Vaillancourt theorem proved in \cite{CardonaRuzhanskyC},  which gives the boundedness  of pseudo-differential operators in the subelliptic calculus in subelliptic Sobolev spaces (see Theorem \ref{Sobolev}).
\end{remark}

\begin{proposition}
Let $A:C^{\infty}(G)\rightarrow \mathscr{D}'(G)$ be a continuous linear operator with symbol $a\in S^{m,\,\mathcal{L}}_{\rho,\delta}(G\times \widehat{G}),$  $0\leq \delta< \rho\leq 1.$ Then $A:H^{s,\,\mathcal{L}}(G)\rightarrow H^{s-m,\,\mathcal{L}}(G) $ extends to a bounded operator for all $s\in \mathbb{R}.$
\end{proposition}
\begin{proof} In view of the closed graph Theorem, we only need to show that there exists $C>0$ such that 
\begin{equation}\label{Sobolev}
    \Vert Au \Vert_{H^{s-m,\,\mathcal{L}}(G)}=\Vert \mathcal{M}_{s-m}Au\Vert_{L^2(G)}\leq C\Vert u \Vert_{H^{s,\,\mathcal{L}}(G)},\,\,u\in C^{\infty}(G).
\end{equation}By replacing $u$ by $\mathcal{M}_{-s}u,$ we can see that \eqref{Sobolev} is equivalent to the following estimate
\begin{equation}\label{Sobolev2}
    \Vert \mathcal{M}_{s-m}A\mathcal{M}_{-s}u\Vert_{L^2(G)}\leq C\Vert u \Vert_{L^2(G)},\,\,u\in C^{\infty}(G),
\end{equation}which, once again by the closed graph theorem, is equivalent to show that $A_{s}:=\mathcal{M}_{s-m}A\mathcal{M}_{-s}$ admits a bounded extension from $C^{\infty}(G)$ to $L^2(G).$ By the subelliptic calculus $A_s\in S^{0,\,\mathcal{L}}_{\rho,\delta}(G\times \widehat{G}).$ This, finally, gives the existence of a bounded extension of $A_{s}$ as a consequence of the subelliptic Calder\'on-Vaillancourt Theorem. 
\end{proof}

\begin{proposition}\label{amplitude}Let $0\leq \delta<\rho\leq 1,$ and let  $a\in \mathcal{A}^{m,\,\mathcal{L}}_{\rho,\delta}(G\times G\times \widehat{G}).$ Then $A\equiv \textnormal{AOp}(a)$ is a subelliptic pseudo-differential operator with symbol $\sigma\in S^{m,\,\mathcal{L}}_{\rho,\delta}(G\times \widehat{G}),$ that is  $A\equiv\textnormal{Op}(\sigma),$ which obeys to the formula
\begin{equation}\label{asy:exp:ampl}
    \sigma(x,\xi)\sim \sum_{\alpha \in \mathbb{N}_0^n}(\partial_{Y}^{(\alpha)}\Delta_\xi^\alpha a(x,y,\xi))|_{y=x},
\end{equation} in the sense that, for all $N\in \mathbb{N}$, and for all $\ell \in\mathbb{N},$
\begin{align}\label{asympexp}
    &\Delta_{\xi}^{\alpha_\ell}\partial_{X}^{(\beta)}\left(\sigma(x,\xi)-\sum_{|\alpha|\leqslant N}  (\partial_{Y}^{(\alpha)}\Delta_\xi^\alpha a(x,y,\xi))|_{y=x}  \right)\in {S}^{m-(\rho-\delta)(N+1)-\rho\ell+\delta|\beta|,\,\mathcal{L}}_{\rho,\delta},
\end{align}for every  $\alpha_\ell \in \mathbb{N}_0^n$ with $|\alpha_\ell|=\ell.$
\end{proposition}
\begin{proof} The proof of the asymptotic expansion \eqref{asy:exp:ampl} is similar to the one for the  analogous statement for elliptic H\"ormander classes. 
Indeed, as in \cite[Page 2891]{RuzhanskyTurunen2011}, we have that, for any amplitude operator $A$,
\begin{equation*}
    \sigma_A(x,\xi)=\int\limits_{G}\xi(z)^{*}\sum_{[\eta]\in \widehat{G}}d_{\eta}\textnormal{Tr}[\eta(z)a(x, xz^{-1} ,\eta)]dz.
\end{equation*}Hence, applying Taylor expansion (see Lemma \ref{Taylorseries}), we have 
\begin{equation*}
    a(x, xz^{-1},\eta)\sim \sum_{\alpha \in \mathbb{N}_0^n}(\partial_{Y}^{(\alpha)}a(x,y,\xi))|_{y=x}q_{(\alpha)}(z),
\end{equation*}and, consequently, we get
\begin{align*}
  \sigma_A(x,\xi)&\sim \sum_{\alpha \in \mathbb{N}_0^n} \int\limits_{G}\xi(z)^{*}\sum_{[\eta]\in \widehat{G}}d_{\eta}\textnormal{Tr}[\eta(z)\partial_{Y}^{(\alpha)}a(x,y,\xi))|_{y=x}q_{(\alpha)}(z)]dz\\
  &= \sum_{\alpha \in \mathbb{N}_0^n} \int\limits_{G}\xi(z)^{*}\sum_{[\eta]\in \widehat{G}}d_{\eta}\textnormal{Tr}[\eta(z)\partial_{Y}^{(\alpha)}a(x,y,\xi))|_{y=x}q_{(\alpha)}(z)]dz\\
  &= \sum_{\alpha \in \mathbb{N}_0^n} \int\limits_{G}\xi(z)^{*}q_{(\alpha)}(z)\sum_{[\eta]\in \widehat{G}}d_{\eta}\textnormal{Tr}[\eta(z)\partial_{Y}^{(\alpha)}a(x,y,\xi))|_{y=x}]dz\\
  &= \sum_{\alpha \in \mathbb{N}_0^n} \int\limits_{G}\xi(z)^{*}q_{(\alpha)}(z)\sum_{[\eta]\in \widehat{G}}d_{\eta}\textnormal{Tr}[\eta(z)\partial_{Y}^{(\alpha)}a(x,y,\xi))|_{y=x}]dz\\
  &=\sum_{\alpha \in \mathbb{N}_0^n}(\partial_{Y}^{(\alpha)}\Delta_\xi^\alpha a(x,y,\xi))|_{y=x}.
\end{align*}Now we analyze the remainder term of the asymptotic expansion above. Our goal is to prove that
\begin{eqnarray}\label{proof:rem:ampl}\Delta_{\xi}^{\alpha_\ell}\partial_{X}^{(\beta)}\left(\sigma(x,\xi)-\sum_{|\alpha|\leqslant N}  (\partial_{Y}^{(\alpha)}\Delta_\xi^\alpha a(x,y,\xi))|_{y=x}  \right)\in {S}^{m-(\rho-\delta)(N+1)-\rho\ell+\delta|\beta|,\,\mathcal{L}}_{\rho,\delta}
\end{eqnarray} for all $\alpha_\ell$  with $|\alpha_\ell|=\ell.$ Below we will only prove the existence of $N_0\in\mathbb{N}$ such that \eqref{proof:rem:ampl} is true for any given $N\geq N_0$. The case when $N<N_0$ follows from the previous one by standard arguments.

For any given $(x,y)\in G\times G$ let $k_{A,x,y}$ be the distribution defined via
$$  a(x,y,\xi)=\widehat{k}_{A,x,y}(\xi),\,\,[\xi]\in \widehat{G}.$$ Denote by $k_{\sigma,x}$ the right-convolution kernel of $A,$ that is the distribution that satisfies $  \sigma(x,\xi)=\widehat{k}_{\sigma,x}(\xi)$ for all $[\xi]\in \widehat{G}.$ Because of the identities
$$ Af(x)\equiv \textnormal{AOp}(a)f(x)=\int\limits_Gf(y)k_{A,x,y}(y^{-1}x)dy=\int\limits_Gf(xz^{-1})k_{A,x,xz^{-1}}(z)dz,  $$
and 
$$ Af(x)\equiv \textnormal{Op}(\sigma)f(x)=\int\limits_Gf(y)k_{\sigma,x}(y^{-1}x)dy=\int\limits_Gf(xz^{-1})k_{\sigma,x}(z)dz,  $$ we have 
$$  k_{A,x,y}(y^{-1}x)=k_{\sigma,x}(y^{-1}x),\quad k_{A,x,xz^{-1}}(z)=k_{\sigma,x}(z), $$
for all $x,y,z\in G.$ Moreover, by using Taylor expansion at $z^{-1}=e$ (see Lemma \ref{Taylorseries}), we can write
$$k_{\sigma,x}(z)= k_{A,x,xz^{-1}}(z)=\sum_{|\alpha|<N}q_{\alpha}(z)\partial_{Z_1}^{(\alpha)}k_{A,xz_1^{-1},x}(z)|_{z_1=e}+R_{x,N}^{k_{A,x,\cdot}}(z), $$ where  the remainder term satisfies the estimate
$$ |R_{x,N}^{k_{A,x,\cdot}}(z)|= |R_{x,N}^{k_{A,x,x}}(z)| \leqslant C|z|^{N}\max_{|\alpha|\leqslant N}\Vert \partial_{X}^{(\alpha)}k_{A,x,x}(z)\Vert_{L_x^\infty(G)}. $$
Observe now that the inverse Fourier transform of the term on the left hand side of \eqref{proof:rem:ampl} is given by
$$ q_{\alpha_\ell}(z)\partial_{X}^{(\beta)}\left(k_{\sigma,x}(z)-\sum_{|\alpha|< N}q_{\alpha}(z)\partial_{Z_1}^{(\alpha)}k_{A,x,xz_1^{-1}}(z)|_{z_1=e} \right)=q_{\alpha_\ell}(z)\partial_{X}^{(\beta)}R_{x,N}^{k_{A,x,x}}(z),   $$
and that, for any given $M'\geq 0$, with $M'\equiv 0 \mod{2}$ and such that
\begin{equation}\label{M'}
    M'>\rho \ell-\delta|\beta|-m+(\rho-\delta)N,
\end{equation} and 
\begin{equation}\label{M'2}
     \frac{M'}{2}>-\rho \ell+\delta|\beta|+m+2(\rho-\delta)N,
\end{equation}
we have
\begin{align*}
  &\Vert\Delta_{\xi}^{\alpha_\ell}\partial_{X}^{(\beta)}\left(\sigma(x,\xi)-\sum_{|\alpha|< N}  (\partial_{Y}^{(\alpha)}\Delta_\xi^\alpha a(x,y,\xi))|_{y=x}  \right)  \widehat{ \mathcal{M}}(\xi)^{(\rho \ell-\delta|\beta|-(m-(\rho-\delta)N))}\Vert_{\textnormal{op}}\\
  &=\Vert\Delta_{\xi}^{\alpha_\ell}\partial_{X}^{(\beta)}\left(\sigma(x,\xi)-\sum_{|\alpha|< N}  (\partial_{Y}^{(\alpha)}\Delta_\xi^\alpha a(x,y,\xi))|_{y=x}  \right)\\
  &\hspace{5cm}\times \widehat{ \mathcal{M}}(\xi)^{M'}\widehat{ \mathcal{M}}(\xi)^{(\rho \ell-\delta|\beta|-(m-(\rho-\delta)N))-M'}\Vert_{\textnormal{op}} \\
   &\leq \Vert\Delta_{\xi}^{\alpha_\ell}\partial_{X}^{(\beta)}\left(\sigma(x,\xi)-\sum_{|\alpha|< N}  (\partial_{Y}^{(\alpha)}\Delta_\xi^\alpha a(x,y,\xi))|_{y=x}  \right) \widehat{ \mathcal{M}}(\xi)^{M'}\|_{\textnormal{op}}\\
   &\lesssim \Vert (1+\mathcal{L}_z)^{\frac{M'}{2}}[q_{\alpha_\ell}(z)\partial_{X}^{(\beta)}R_{x,N}^{k_{A,x,x}}(z)] \Vert_{L^{1}(G,dz)}.\\
\end{align*}
Now,  the application of Leibniz rule gives 
\begin{align*}
     &\Vert (1+\mathcal{L}_z)^{\frac{M'}{2}}[q_{\alpha_\ell}(z)\partial_{X}^{(\beta)}R_{x,N}^{k_{A,x,x}}(z)] \Vert_{L^{1}(G,dz)}=\Vert (1+\mathcal{L}_z)^{\frac{M'}{2}}[R_{x,N+1}^{q_{\alpha_\ell}(z)\partial_{X}^{(\beta)}k_{A,x,x}}(z)] \Vert_{L^{1}(G,dz)}\\
     &\lesssim \sum_{1\leqslant i_1\leqslant i_2\leqslant\cdots \leqslant i_{k}\leqslant k\,,|\gamma|\leqslant M'}\Vert X^{\gamma_1}_{i_1,z}\cdots X^{\gamma_k}_{i_k,z} [R_{x,N}^{q_{\alpha_\ell}(z)\partial_{X}^{(\beta)}k_{A,x,x}}(z)] \Vert_{L^{1}(G,dz)}.
\end{align*}   
By the estimates in Lemma  \ref{Taylorseries} we have
\begin{align*}
   | X^{\gamma_1}_{i_1,z}\cdots X^{\gamma_k}_{i_k,z} [R_{x,N+1}^{q_{\alpha_\ell}(z)\partial_{X}^{(\beta)}k_{A,x,x}}(z)]|\lesssim |z|^{N-|\gamma|}\max_{|\alpha|\leq N-|\gamma|}\Vert \partial_{Z}^{(\alpha+\gamma)}( {q_{\alpha_\ell}(z)\partial_{X}^{(\beta)}k_{A,x,x}}(z))\Vert_{L^{\infty}(G)},
\end{align*}
and, since $ \partial_{Z}^{(\gamma)} ({q_{\alpha_\ell}(z)\partial_{X}^{(\beta)}k_{A,x,x}}(z))$ is the right-convolution kernel of a subelliptic operator of order $$s'=m+\delta|\beta|+\delta|\alpha|+\delta|\gamma|-\rho\ell,$$ by Proposition 4.24 of \cite{CardonaRuzhanskyC} we have
$$ \max_{|\alpha|\leq M'}\Vert \partial_{Z}^{(\alpha)} ({q_{\alpha_\ell}(z)\partial_{X}^{(\beta)}k_{A,x,x}}(z))\Vert_{L^{\infty}(G)}\lesssim_{a}|z|^{-\frac{s'+Q}{\rho}}.$$
 Therefore, putting the previous estimates together, we obtain
 \begin{align}
     &\Vert (1+\mathcal{L}_z)^{\frac{M'}{2}}[q_{\alpha_\ell}(z)\partial_{X}^{(\beta)}R_{x,N}^{k_{A,x,x}}(z)] \Vert_{L^{1}(G,dz)}\lesssim_{a} \int\limits_G|z|^{N-|\gamma|}|z|^{-\frac{s'+Q}{\rho}}dz, 
 \end{align}where the integral on the right hand side is finite 
 if  
 $$ \rho(|\gamma|-N)+s'+Q<n{\rho}. $$
Notice that, for $N$ big enough, we have $(\rho-\delta)N\lesssim M'$. However, by choosing $M'$ in such a way that $M'\lesssim 100(\rho-\delta)N, $ we have that $(\rho-\delta)N\asymp  M'.$
 Hence, by using  the inequality $|\alpha|+|\gamma|\leq N$ together with \eqref{M'2}, we have
 \begin{align*}
  &\rho(|\gamma|-N) +   s'+Q =\rho|\gamma|-\rho N+ m+\delta|\beta|+\delta|\alpha|+\delta|\gamma|-\rho\ell+Q\\
  &<\rho|\gamma|-\rho N+ m+\delta|\beta|+\rho|\alpha|+\rho|\gamma|-\rho\ell+Q\\
  &=\rho|\gamma|-\rho (N-|\alpha|-|\gamma|)+ m+\delta|\beta|-\rho\ell+Q\leq \rho|\gamma|+ m+\delta|\beta|-\rho\ell+Q\\
  &\leq  \rho M'+\frac{M'}{2}-2N(\rho-\delta)+Q\leq \frac{3M'}{2}-2N(\rho-\delta)+Q\\
  &\asymp \frac{3(\rho-\delta)N}{2}-2N(\rho-\delta)+Q= -\frac{(\rho-\delta)N}{2}+Q.
 \end{align*} Now, by choosing  $N\gtrsim N_0\geq 2( n\rho-Q)/(\rho-\delta))$, we get that $|\gamma|-N+   s'+Q<n\rho$, which, in particular, gives the validity of \eqref{asympexp} for $N$ large enough. Finally, the case $N\lesssim N_0$ follows from standard arguments from the previous case. This concludes the proof.
\end{proof}

\section{Proof of the main Theorem \ref{MainTheorem}}\label{Sect3}
{\bf Notations.} In this section the following notations will be adopted:
\begin{itemize}
 \item[-] $q_\gamma(x):=q_{(1)}^{\gamma_1}(x)\ldots q_{(n)}^{\gamma_n}(x)$, for any given $\gamma\in \mathbb{N}_0^n$;
    \item [-] $X=\{X_{1},\cdots,X_{k}\}$ will be a system of vector fields satisfying H\"ormander condition of step $\kappa;$ 
    \item [-] $\mathcal{L}:=\mathcal{L}_X$ will be the positive sub-Laplacian associated with the system $X$;
    \item [-] $\mathcal{M}$ will be the operator defined as $\mathcal{M}:=(I+\mathcal{L})^{1/2}$, while  $\widehat{\mathcal{M}}(\xi)$ will denote the corresponding matrix-valued symbol  $\widehat{\mathcal{M}}(\xi)=\mathrm{diag}[\langle \nu_{ii}(\xi)\rangle]_{1\leq i\leq d_\xi},$ where $\langle \nu_{ii}(\xi)\rangle :=(1+\nu_{ii}(\xi)^2)^{1/2}$;
    \item [-]$\|\cdot\|_{m,\,\mathcal{L}}$ will denote the norm in the Sobolev space $H^{m,\,\mathcal{L}}(G)$ associated with the subelliptic operator $\mathcal{L}$ on $G$;
    \item[-] $\|\cdot\|_{op}$ will denote the $\ell^2\rightarrow\ell^2$ operator norm of the linear finite dimensional mapping (matrix multiplication by) $\sigma_A(x,\xi)$, that is,
    \begin{equation*}
        \|\sigma_A(x,\xi)\|_{op}=\sup \{ \|\sigma_A(x,\xi)v\|_{\ell^2}: v\in \mathbb{C}^{d_\xi},\|v\|_{\ell^2}=1\}.
    \end{equation*}
\end{itemize}
\begin{remark}\label{SerenaRemark2}
The proof of the main theorem relies on the construction of a positive  pseudo-differential operator $P\in \textnormal{Op}(S^{m,\,\mathcal{L}}_{\rho,\frac{\rho+\delta}{2}}(G\times \widehat{G}))$ such that $A-P=Q,$ with $A$ as in the hypotheses of Theorem \ref{MainTheorem}, and $$Q: H^{\frac{m-\frac 1\kappa(\rho-(2\kappa-1)\delta)}{2},\,\,\mathcal{L}}(G)\rightarrow H^{\frac{-(m-\frac 1 \kappa(\rho-(2\kappa-1)\delta))}{2},\,\,\mathcal{L}}(G)$$ being a  bounded operator. In fact, if such a decomposition were true, then this would immediately lead to the desired inequality by simply observing that
\begin{equation*}
    { \mathrm{Re}}(Au,u)=\mathrm{Re}(Pu,u)+\mathrm{Re}(Qu,u)\geq -C\| u\|_{\frac{m-\frac 1 \kappa(\rho-(2\kappa-1)\delta)}{2},\,\,\mathcal{L}},
\end{equation*}
for all $u\in C^\infty(G)$.
\end{remark}

As in the proof of the Sharp G\r{a}rding inequality on compact Lie groups (see \cite{RuzhanskyTurunen2011}), the key point here is the construction of the so called {\it weight} function $w_\xi$, which, essentially, corresponds to the construction of the operator $P$ with the properties mentioned in Remark \ref{SerenaRemark2}.

The aforementioned construction follows the lines of the Euclidean case. However, the adaptation to the Lie-group setting is linked to the deep group structure. Indeed, a  crucial key point in our proof will be the choice of the suitable power of $\langle \xi\rangle$ inside the expression of $w_\xi$, choice that depends on the relation between the eigenvalues of the Laplacian and those of the sublaplacian on the group.

We remark that the use of an {\it elliptic} weight $w_\xi$, namely depending on $\langle \xi \rangle$, instead of a subelliptic one, that is depending on $\widehat{\mathcal{M}}(\xi)$, is fundamental here. This, in particular, is due to the fact that while the Laplacian is a central operator on the group and its (matrix valued) symbol commutes with all the other symbols, the same property does not hold for any sublaplacian on the group. The noncommutativity property of the symbol of sub-Laplacians, toghether with the intrinsic noncommutativity of the group, causes several technical issues not allowing the use of a subelliptic weight.

Notice also that similar anomalies occur in the nilpotent Lie group setting, where, again, sub-Laplacians (and more generally Rockland operators), do not have symbols with the same commutaivity property as the symbol of the Laplacian on compact Lie groups. This main difference, as remarked in \cite{FisherRuzhansky2013}, is very likely the reason why the full sharp G\r{a}rding inequality, that is for any operator with nonnegative symbol, has not been proved yet in the nilpotent Lie group setting (see \cite{FisherRuzhansky2013} where the result for specific operators has been announced). 

Taking into account the previous clarifications, we can now build the proof of Theorem \ref{MainTheorem} starting from the construction of the weight function $w_\xi$.

Let us consider $G$ as a closed subgroup of $GL(N,\mathbb{R})\subset\mathbb{R}^{N\times N}$ for some $N\in\mathbb{N}$, so that its Lie algebra $\mathfrak{g}$ is an $n$-dimensional vector subspace such that $$[A,B]:=AB-BA\in \mathfrak{g}$$ for every $A,B\in\mathfrak{g}$. Let $e\in G$ be the neutral element, $U\subset G$ a neighborhood of $e$, and $V\subset \mathfrak{g}$ a neighborhood of $0\in\mathfrak{g}$ such that the matrix exponential mapping $$\exp:V \rightarrow U$$ is a diffeomorphism. We  define on $\mathfrak{g}$ the central norm $|\cdot|$ (that we shall use only in the definition of the function $w_\xi$ below) as follows
\begin{equation}\label{norm}
    |X|:=\int\limits_G|uXu^{-1}|_0du,
\end{equation}
where the product under the integral is the product of matrices, and where $|\cdot|_0$ stands for the Euclidean norm on $\mathfrak{g}$.
Note that $\exp^{-1}$ is central with respect to the norm \eqref{norm}, i.e., $|\exp^{-1}(xy)|=|\exp^{-1}(yx)|$, and that the norm \eqref{norm} is invariant by the adjoint representation.

We now assume that the neighborhood $V$ of $0_\mathfrak{g}$ is the open ball $V=\mathbb{B}(0,r)=\{Z\in \mathbb{R}^n: |Z|<r\}$, with $r>0$, and consider a real smooth function $\phi:[0,r)\rightarrow [0,\infty)$, radial on $\mathfrak{g}$, supported on $V$, and such that $\phi(s)=\phi(|Z|)=1$ for $s>0$ small. Then we define the function
\begin{equation}\label{w}
    w_\xi(x):=\phi(|\exp^{-1}(x)|\langle \xi\rangle^{\frac{(\rho+\delta)}{2\kappa}})\psi(\exp^{-1}(x))\langle\xi\rangle^{\frac{n (\rho+\delta)}{4\kappa}}
\end{equation}
with  
\begin{equation*}
    \psi(Y)=C_0|\det D \exp(Y)|^{-1/2}f(Y)^{-1/2},\quad C_0=(\int\limits_{\mathbb{R}^n}\phi(|Z|)^2 dZ)^{-1/2},
\end{equation*}
where, recall, $\rho+\delta<2$, $\langle \xi\rangle=(1+\lambda_\xi^2)^{\frac{1}{2}}$, i.e. $\sigma_{I-\Delta}(\xi)=\langle \xi\rangle^2 I_{d_\xi}$,
$D \exp$ is the Jacobian matrix of $\exp$, and $f(Y)$ is the density with respect to the Lebesgue measure of the pull-back on $\mathfrak{g}$ of the Haar measure on $G$ by the exponential mapping.

Note that, given the geodesic distance $d$, we have $d(x,e)\simeq |\exp^{-1}(x)|$ and $\mathrm{supp}\,{w_\xi}= \{x\in G: d(x,e)\leq r \langle \xi\rangle^{-\frac{(\rho+\delta)}{2\kappa}} \}\subset U_e$. 
\begin{remark}\label{Serena:Remark} Observe  that from Remark \ref{SerenaRemark} and the fact that $\langle\xi\rangle^{\frac{n(\rho+\delta)}{4\kappa}}I_{d_\xi}\in  \mathscr{S}^{\frac{n(\rho+\delta)}{4\kappa}}_{1,0}(G),$  we deduce that $\langle\xi\rangle^{\frac{n(\rho+\delta)}{4\kappa}}I_{d_\xi}\in S^{\frac{n(\rho+\delta)}{4},\,\mathcal{L}}_{1,0}(G\times \widehat{G}) \subset S^{\frac{n(\rho+\delta)}{4},\,\mathcal{L}}_{\rho,0}(G\times \widehat{G})$, for all $\rho\in (0,1]$.

\end{remark}

We are now ready to state a key lemma about the basic properties of the function $w_\xi$.

\begin{lemma}
Let $w_\xi$ be as above. Then $w_\xi(e)=C_0 \langle \xi\rangle^{\frac{n(\rho+\delta)}{4\kappa}}$, $w_\xi$ is central and inversion invariant, and $d(x,e)\leq r \langle \xi\rangle^{-\frac{(\rho+\delta)}{2\kappa}}$ on $\mathrm{supp}\, (w_\xi)$. Moreover, 
$\| w_\xi\|_{L^2(G)}= 1$ for all $\xi\in \mathrm{Rep}(G),$ and $(x,\xi)\mapsto w_\xi(x)I_{d_\xi}\in \mathscr{S}^{\frac{n(\rho+\delta)}{4\kappa}}_{1,\frac{(\rho+\delta)}{2\kappa}}(G)\subset S^{\frac{n(\rho+\delta)}{4},\,\mathcal{L}}_{\rho,\frac{(\rho+\delta)}{2}}(G\times \widehat{G})$, for all $\rho \in (0,1]$ and $0\leq \delta<\rho$.
\end{lemma}

\begin{proof}
Notice that, due to the properties of $\phi$, we immediately have that $w_\xi(e)=C_0 \langle \xi\rangle^{\frac{n(\rho+\delta)}{4\kappa}}$, $w_\xi$ is central and inversion invariant, and $d(x,e)\leq r \langle \xi\rangle^{-\frac{(\rho+\delta)}{2\kappa}}$ on $\mathrm{supp}\, w_\xi$. 
As for the $\| w_\xi \|_{L^2(G)}$, we have
\begin{equation*}
    \begin{split}
        \int\limits_G |w_\xi(x)|^2dx&= \langle \xi\rangle^{\frac{n(\rho+\delta)}{2\kappa}}\int\limits_{\mathbb{R}^n}\phi_\xi(|Y|\langle \xi\rangle^{\frac{(\rho+\delta)}{2\kappa}})^2|\psi(|Y|)|^2|det D \exp(Y)| f(Y) dY\\
        &= \int\limits_{\mathbb{R}^n}\phi_\xi(|Z|)^2|\psi(|Z|\langle\xi\rangle^{-\frac{(\rho+\delta)}{2\kappa}})|^2|det D \exp(Z\langle\xi\rangle^{-\frac{(\rho+\delta)}{2\kappa}})| f(Z\langle\xi\rangle^{-\frac{(\rho+\delta)}{2\kappa}}) dZ\\
        &=C_0^2\int\limits_{\mathbb{R}^n}\phi_\xi(|Z|)^2dZ=1,
    \end{split}
\end{equation*}
where in the second line we applied the change of variables $Z=Y\langle\xi\rangle^{\frac{(\rho+\delta)}{2\kappa}}$, while in the third line we simply used the expression of $\psi$.

We are now left with the proof of 
\begin{equation}\label{StandardS}
    (x,\xi)\mapsto w_\xi(x)I_{d_\xi}\in \mathscr{S}^{\frac{n(\rho+\delta)}{4\kappa}}_{\rho,\frac{(\rho+\delta)}{2\kappa}}(G)\subset S^{\frac{n(\rho+\delta)}{4},\,\mathcal{L}}_{\rho,\frac{(\rho+\delta)}{2}}(G\times \widehat{G}).
    \end{equation}
From the compactness of $G$, proving \eqref{StandardS} is equivalent to showing that, for every multi-index $\beta$ and for any fixed $x\in G$, $(\partial^{(\beta)}_X w_\xi)(x)I_{d_\xi}\in \mathscr{S}_{\rho,0}^{\frac{n(\rho+\delta)}{4\kappa}+\frac{(\rho+\delta)}{2\kappa}|\beta|}(G)$ (see Lemma 3.3 in \cite{RuzhanskyTurunen2011}).
First observe that 
$$\partial^{(\beta)}_X  w_\xi(x)I_{d_\xi}= \sum_{\alpha;|\alpha|\leq|\beta|} C_{\alpha,\beta}\Big[\partial_X^{(\alpha)}\phi(|\exp^{-1}(x)|\langle \xi\rangle^{\frac{(\rho+\delta)}{2\kappa}})\Big]\partial_X^{(\beta-\alpha)}\psi(\exp^{-1}(x))\langle\xi\rangle^{\frac{n(\rho+\delta)}{4\kappa}}I_{d_\xi}$$
$$= \sum_{\alpha;|\alpha|\leq|\beta|} C_{\alpha,\beta}\phi_{\alpha}(|\exp^{-1}(x)|\langle \xi\rangle^{\frac{(\rho+\delta)}{2\kappa}})\langle \xi\rangle^{\frac{(\rho+\delta)|\alpha|}{2\kappa}}\chi_{\beta-\alpha}(\exp^{-1}(x))\langle\xi\rangle^{\frac{n(\rho+\delta)}{4\kappa}}I_{d_\xi},$$
where $\phi_{\alpha},\chi_{\beta-\alpha}$ are suitable functions such that $\phi_{\alpha}\in C_0^\infty(\mathbb{R})$ and is constant near the origin, while $\chi_{\beta-\alpha}\in C_0^\infty(V)$.
Since $$\langle\xi\rangle^{\frac{n(\rho+\delta)}{4\kappa}+\frac{(\rho+\delta)|\alpha|}{2\kappa}}I_{d_\xi}\in \mathscr{S}^{\frac{n(\rho+\delta)}{4\kappa}+\frac{(\rho+\delta)}{2\kappa}|\beta|}_{1,0}(G),$$ 
then, for every (fixed) $x\in G$, $$\partial^\beta_x w_\xi(x)I_{d_\xi}\in \mathscr{S}_{1,0}^{\frac{n(\rho+\delta)}{4\kappa}+\frac{(\rho+\delta)}{2\kappa}|\beta|}(G)$$ 
if 
$\phi_\alpha(|\exp^{-1}(x)|\langle \xi\rangle^{\frac{(\rho+\delta)}{2\kappa}})\chi_{\beta-\alpha}(\exp^{-1}(x)) \in \mathscr{S}^{0}_{1,0}(G)$ for all $\alpha$ and $\beta$ as above. Therefore, to complete the proof it is enough to check that these last terms are standard global symbols of order 0.

Now, given $x\in G$, it is easy to see that
\begin{equation}
\label{est.lemma1}
\phi_{\alpha}(|\exp^{-1}(x)|\langle \xi\rangle^{\frac{(\rho+\delta)}{2\kappa}})\chi_{\beta-\alpha}(\exp^{-1}(x))\leq C.
\end{equation}
In fact, if $x$ is such that $\exp^{-1}(x)=0$, then $\phi_{\alpha}$ is constant and the inequality follows. If, instead, $\exp^{-1}(x)\neq0$, then, since $\phi_{\alpha}$ is compactly supported in $\xi$, we get that the symbol in the left hand side of \eqref{est.lemma1} is compactly supported, then smoothing, and the inequality follows. This concludes the proof of \eqref{StandardS} from which the result follows.
\end{proof}



\begin{remark}
Note that from Remark \ref{SerenaRemark}, for all $m>0$, we have  $\mathscr{S}^{\frac{m}{\kappa}}_{\rho,\frac{\delta}{\kappa}}(G)\subset S^{m,\,\mathcal{L}}_{\rho,\delta}(G\times \widehat{G})$. This fact, together with the property $(x,\xi)\mapsto w_\xi(x)I_{d_\xi}\in \mathscr{S}^{\frac{n(\rho+\delta)}{4\kappa}}_{\rho,\frac{(\rho+\delta)}{2\kappa}}(G)$, immediately gives  $(x,\xi)\mapsto w_\xi(x)I_{d_\xi}\in S^{\frac{n(\rho+\delta)}{4},\,\mathcal{L}}_{\rho,\frac{(\rho+\delta)}{2}}(G)$. 
Note also that the symbol $w_\xi(x)I_{d_\xi}$ commutes with any subelliptic symbol $b(x,\xi)$, for all $\xi\in\widehat{G}$.
\end{remark}

\begin{proposition}\label{P}
Let $\sigma_A\in S^{m,\,\mathcal{L}}_{\rho,\delta}(G\times \widehat{G})$ and let $p(x,y,\xi)$ be the amplitude
\begin{equation}\label{p.amplitude}
   p(x,y,\xi):=\int\limits_G w_\xi(xz^{-1}){w_\xi}(yz^{-1}) \sigma_A(z,\xi)dz,
\end{equation}
where $w_\xi$ is as in \eqref{w}. Then $p\in \mathcal{A}^{m,\,\mathcal{L}}_{\rho,\frac{(\rho+\delta)}{2}}(G\times G\times \widehat{G})$ and the amplitude operator $P=\textnormal{Op}(p)$ given by
$$Pu(x)=\int\limits_G \sum_{[\xi]\in\widehat{G}}d_{\xi}\mathrm{Tr}\left( \xi(y^{-1}x)p(x,y,\xi)\right)u(y)dy $$
is positive.
\end{proposition}
\begin{proof}
Recall that $p\in \mathcal{A}^{m,\,\mathcal{L}}_{\rho,\frac{(\rho+\delta)}{2}}(G\times G\times \widehat{G})$ if
$$ \sup_{(x,y,[\xi])\in G\times G\times  \widehat{G} }\Vert \widehat{ \mathcal{M}}(\xi)^{-m+\rho|\alpha|-\frac{(\rho+\delta)}{2}(|\beta|+|\gamma|)}\partial_{X}^{(\beta)}\partial_{Y}^{(\gamma)} \Delta_{\xi}^{\alpha}p(x,y,\xi)\Vert_{\textnormal{op}} <\infty,$$
and that, by Leibniz rule, $\partial_{X}^{(\beta)}\partial_{Y}^{(\gamma)} \Delta_{\xi}^{\alpha}p(x,y,\xi)$ is a sum of terms of the form
$$ \int\limits_{G}( \Delta_\xi^{\eta} \partial_X^{(\beta)} w_\xi(xz^{-1}))( \Delta_\xi^{\lambda}  \partial_Y^{(\gamma)} w_\xi(yz^{-1}))(\Delta_\xi^\mu\sigma_A(z,\xi))dz,$$
where $|\eta+\lambda+\mu|\geq |\alpha|$. 
Moreover, due to the properties of $w_\xi$, we have
$$\|( \Delta_\xi^{\eta} \partial_X^{(\beta)} w_\xi(xz^{-1}))( \Delta_\xi^{\lambda}  \partial_Y^{(\gamma)} w_\xi(yz^{-1}))\|_{op}$$
$$\leq C  \langle \xi \rangle^{\frac{n(\rho+\delta)}{4\kappa}-\rho|\eta|+\frac{(\rho+\delta)}{2}\frac{|\beta|}{\kappa}}\langle\xi\rangle^{\frac{n(\rho+\delta)}{4\kappa}-\rho|\lambda|+\frac{(\rho+\delta)}{2}\frac{|\gamma|}{\kappa}}$$
$$\leq C\langle \xi \rangle^{\frac{n(\rho+\delta)}{2\kappa}}\sup_{i=1,...,d_\xi}\langle\nu_{ii}(\xi)\rangle^{-\rho(|\lambda|+|\eta|)+\frac{(\rho+\delta)}{2}(|\beta|+|\gamma|)},$$
therefore, since $\mathrm{supp}\, w_\xi=\{z; d(z,e)\leq r\langle \xi\rangle^{-\frac{(\rho+\delta)}{2\kappa}}\}$ is contained in a set of measure $r\langle \xi \rangle^{-\frac{(\rho+\delta) n}{2\kappa}}$,  we get 

$$ \|\int\limits_{G}\widehat{\mathcal{M}}(\xi)^{-\rho|\mu|+\rho|\alpha|-\frac{(\rho+\delta)}{2}(|\beta|+|\gamma|)}( \Delta_\xi^{\eta} \partial_X^{(\beta)} w_\xi(xz^{-1}))( \Delta_\xi^{\lambda}  \partial_Y^{(\gamma)} w_\xi(yz^{-1}))$$
$$\times \widehat{\mathcal{M}}(\xi)^{-m+\rho|\mu|}(\Delta_\xi^\mu\sigma_A(z,\xi))dz\|_{op}$$
$$\leq C\langle\xi\rangle^{-\frac{n(\rho+\delta)}{2\kappa}}\| \Delta_\xi^{\eta} \partial_X^{(\beta)} w_\xi(xz^{-1}) \Delta_\xi^{\lambda}  \partial_Y^{(\gamma)} w_\xi(yz^{-1}))\widehat{\mathcal{M}}(\xi)^{-\rho|\mu|+\rho|\alpha|-
  \frac{(\rho+\delta)}{2}(|\beta|+|\gamma|)}
 \|_{op}$$
$$\times \|\widehat{\mathcal{M}}(\xi)^{-m+\rho|\mu|}(\Delta_\xi^\mu\sigma_A(z,\xi))\|_{op}$$
$$\leq C \langle\xi\rangle^{-\frac{n(\rho+\delta)}{2\kappa}} \langle\xi\rangle^{\frac{n(\rho+\delta)}{2\kappa} }
	\|\langle \xi\rangle^{-\rho(|\lambda|+|\eta|)+\frac{(\rho+\delta)}{2\kappa}(|\beta|+|\gamma|)}\widehat{\mathcal{M}}(\xi)^{-\rho|\mu|+\rho|\alpha|-\frac{(\rho+\delta)}{2}(|\beta|+|\gamma|)}\|_{op}$$

$$\leq C\sup_{i=1,...,d_\xi}\langle\nu_{ii}(\xi)\rangle^{\rho|\alpha|-\frac{(\rho+\delta)}{2}(|\beta|+|\gamma|)-\rho(|\eta|+|\lambda|+|\mu|)+\frac{(\rho+\delta)}{2}(|\beta|+|\gamma|)}
\leq C.$$
Finally, since $\Vert \widehat{ \mathcal{M}}(\xi)^{-m+\rho|\alpha|-\frac{(\rho+\delta)}{2}(|\beta|+|\gamma|)}\partial_{X}^{(\beta)}\partial_{Y}^{(\gamma)} \Delta_{\xi}^{\alpha}p(x,y,\xi)\Vert_{\textnormal{op}}$ is estimated by a sum of terms of the previous form, the first result follows.

To see that $P$ is positive, on denoting 
$$ M(z,\xi):=\int\limits_G w_\xi(yz^{-1})\xi(y z^{-1}) \overline{u(y)}dy,$$
we have
$$ (Pu,u)=\int\limits_G \int\limits_G \sum_{[\xi]\in\widehat{G}} d_\xi \mathrm{Tr}(\xi(x)p(x,y,\xi)u(y)\xi(y)^*dy) \overline{u(x)}dx$$

$$=\int\limits_G \int\limits_G \sum_{[\xi]\in\widehat{G}} d_\xi \mathrm{Tr}\left(\xi(x)   
\int\limits_G w_\xi(xz^{-1}){w_\xi}(yz^{-1})\sigma_A(z,\xi)dz
u(y)\xi(y)^*dy \right)\overline{u(x)}dx$$

$$=\int\limits_G \sum_{[\xi]\in\widehat{G}} d_\xi \mathrm{Tr}\left(  
\left(\int\limits_G \xi(xz^{-1}) w_\xi(xz^{-1})\overline{u(x)}dx\right) \xi(z) \sigma_A(z,\xi)\xi(z)^*\right.$$ $$\left.\times\left(\int\limits_{G}{w_\xi}(yz^{-1})\xi(yz^{-1})^* u(y)dy \right) 
 \right)dz$$

$$=\int\limits_G \sum_{[\xi]\in\widehat{G}} d_\xi \mathrm{Tr}\left(M(z,\xi) \xi(z)\sigma_A(x,\xi) \xi(z)^*M(z,\xi)^* \right)dz, $$
which is non-negative due to the non-negativity of $\sigma_A$ and the fact that the positivity of matrices is invariant under unitary transformations. This concludes the proof.
\end{proof}

\begin{lemma}\label{Lemma1}
Let $s\in\mathbb{R}$ and $p\in \mathcal{A}^{m,\,\mathcal{L}}_{\rho,\frac{(\rho+\delta)}{2}}(G\times G\times \widehat{G})$ as in \eqref{p.amplitude}. Let us assume that $0\leq\delta< \rho\leq 1$. Then the operator with symbol $p(x,x,\xi)-\sigma_A(x,\xi)$ is bounded from $H^{s,\,\mathcal{L}}(G)$ to $H^{s-(m-\frac 1 \kappa(\rho-(2\kappa-1)\delta)),\,\mathcal{L}}(G)$.
\end{lemma}

\begin{proof}
By the properties of pseudo-differential operators with symbols in subelliptic classes, in order to prove the theorem it suffices to show that
\begin{equation}\label{CV}
p(x,x,\xi)-\sigma_A(x,\xi)\in S^{m-\frac 1 \kappa(\rho-(2\kappa-1)\delta),\,\mathcal{L}}_{
	\rho ,\max\{\delta,\frac{(\rho+\delta)}{2}\}}(G\times \widehat{G})=S^{m-\frac 1 \kappa(\rho-(2\kappa-1)\delta),\,\mathcal{L}}_{
	\rho ,\frac{(\rho+\delta)}{2}}(G\times \widehat{G}).
\end{equation}
Note that, since $\|w_\xi\|_{L^2(G)}=1$, we have
$$p(x,x,\xi)-\sigma_A(x,\xi)=\int\limits_G w_\xi(xz^{-1})^2\sigma_A(z,\xi)dz-\sigma_A(x,\xi)$$
$$=\int\limits_G w_\xi(z)^2(\sigma_A(z^{-1}x,\xi)-\sigma_A(x,\xi))dz,$$
where, recall, $\mathrm{supp}\, w_\xi =\{z\in G: d(z,e)\leq r \langle \xi\rangle^{-\frac{(\rho+\delta)}{2\kappa}}\}$. Then we consider the Taylor expansion of order one of $\sigma_A(z^{-1}x,\xi)$ with respect to $z^{-1}$ at $z^{-1}=e$, namely,
$$\sigma_A(z^{-1}x,\xi)=\sigma_A(x,\xi)+\sum_{|\gamma|=1}\partial_X^{(\gamma)}\sigma_A(x,\xi)q_\gamma(z)+\sum_{|\gamma|=2}q_\gamma(z)\sigma_{A,\gamma}(z^{-1}x,\xi),$$
 where the last term in the right hand side represents the remainder of Taylor expansion of order 1. Note also that we can choose the polynomials $q_\gamma$ to be odd when $|\gamma|=1$, that is, $q_\gamma(z)=-q_\gamma(z^{-1})$.
By using the expansion above together with the property $\int\limits_G|w_\xi(z)|^2q_\gamma(z)dz=0$ when $|\gamma|=1$ (since $w_\xi$ is even while $q_\gamma$ is odd), we obtain
$$p(x,x,\xi)-\sigma_A(x,\xi)=\int\limits_G w_\xi(z)^2\Big(\sigma_A(x,\xi)-\sigma_A(x,\xi)$$
$$+\sum_{|\gamma|=1}\partial_X^{(\gamma)}\sigma_A(x,\xi)q_\gamma(z)+\sum_{|\gamma|=2}\sigma_{A,\gamma}(z^{-1}x,\xi)q_\gamma(z)\Big)dz$$
$$=\sum_{|\gamma|=2}\int\limits_G \sigma_{A,\gamma}(z^{-1}x,\xi)w_\xi(z)^2q_\gamma(z)dz,$$
where the symbols $\sigma_{A,\gamma}\in S^{m+\delta|\gamma|,\,\mathcal{L}}_{\rho,\delta}(G\times \widehat{G})$ come from the remainder term of Taylor expansion.
Now, by Leibniz rule and the left invariance of $\partial^{(\beta)}_X$, we have that the quantity
$$\partial^{(\beta)}_X\triangle_\xi^\alpha (p(x,x,\xi)-\sigma_A(x,\xi))
=\partial^{(\beta)}_X\triangle_\xi^\alpha \sum_{|\gamma|=2}\int\limits_G \sigma_{A,\gamma}(z^{-1}x,\xi)w_\xi(z)^2q_\gamma(z)dz$$
can be written as a sum of terms of the form
 $$ \int\limits_G (\partial^{(\beta)}_X\triangle_\xi^{\alpha_1}\sigma_{A,\gamma}(z^{-1}x,\xi))(\triangle_\xi^{\alpha_2}w_\xi(z))(\triangle_\xi^{\alpha_3} w_\xi(z))q_\gamma(z)dz$$
\begin{equation}\label{remainder1}
 = \int\limits_G (\partial^{(\beta)}_{Y}\triangle_\xi^{\alpha_1}\sigma_{A,\gamma}(y,\xi))|_{y=z^{-1}x}(\triangle_\xi^{\alpha_2}w_\xi(z))(\triangle_\xi^{\alpha_3} w_\xi(z))q_\gamma(z)dz,
\end{equation}
where $|\gamma|=2$ and $|\alpha_1+\alpha_2+\alpha_3|\geq |\alpha|$.

Since for $\gamma$ such that $|\gamma|=2$ we have $|q_\gamma(z)|\leq C\langle \xi\rangle^{-\frac{(\rho+\delta)}{\kappa}}$ on the support of $ w_\xi$, and since $\mathrm{meas}(\mathrm{supp}( w_\xi))\lesssim \langle \xi\rangle^{-\frac{n(\rho+\delta)}{2\kappa}}$, in view of \eqref{GarettoRuzhanskyIneq} each therm of the form \eqref{remainder1} satisfies

$$ \| \widehat{\mathcal{M}}(\xi)^{-m+\frac{1}{\kappa}(\rho-(2\kappa-1)\delta)+\rho|\alpha|-\max\{\delta, \frac{(\rho+\delta)}{2} \}|\beta|}\int\limits_G (\partial^{(\beta)}_{Y}\triangle_\xi^{\alpha_1}\sigma_{A,\gamma}(y,\xi))|_{y=z^{-1}x}(\triangle_\xi^{\alpha_2}w_\xi(z))$$
$$\times (\triangle_\xi^{\alpha_3} w_\xi(z))q_\gamma(z)dz\|_{op}$$
$$ \leq C \sup_{x\in G}\| \widehat{\mathcal{M}}(\xi)^{-m+\frac{1}{\kappa}(\rho-(2\kappa-1)\delta)+\rho|\alpha|-\frac{(\rho+\delta)}{2}|\beta|}
\partial^{(\beta)}_{X}\triangle_\xi^{\alpha_1}\sigma_{A,\gamma}(x,\xi))
\|_{op} $$
$$\sup_{z_1\in G} \|\triangle_\xi^{\alpha_2} w_\xi(z_1)\|_{op}
\sup_{z_2\in G} \|\triangle_\xi^{\alpha_3} w_\xi(z_2)\|_{op} \int\limits_{\mathrm{supp}\, (w_\xi)} |q_\gamma(z)|\,dz$$
$$\leq C \sup_{i=1,...,d_\xi}\langle\nu(\xi)_{ii}\rangle^{-m+\frac{1}{\kappa}(\rho-(2\kappa-1)\delta)+\rho|\alpha|-\delta|\beta|+m+2\delta+|\beta|\delta-|\alpha_1|\rho}
$$
$$\times \,\, \langle \xi \rangle^{\frac{n(\rho+\delta)}{2\kappa}-(|\alpha_2|+|\alpha_3|)\rho-\frac{n(\rho+\delta)}{2\kappa}-\frac{(\rho+\delta)}{\kappa}} $$
$$\leq C \sup_{i=1,...,d_\xi}\langle\nu(\xi)_{ii}\rangle^{\frac{1}{\kappa}(\rho-(2\kappa-1)\delta)+\rho|\alpha|-\delta|\beta|+2\delta+|\beta|\delta-(|\alpha_1|+|\alpha_2|+|\alpha_3|)\rho-\frac{(\rho+\delta)}{\kappa}}\leq C.
$$
This, finally,  shows \eqref{CV} and concludes the proof.
\end{proof}

\begin{lemma}\label{Lemma2}
Let $s\in \mathbb{R}.$  Let us assume that  $0< \rho\leqslant 1$ and  $0\leq \delta<(2\kappa-1)^{-1}\rho\leq \rho.$ Then the pseudo-differential operator with symbol $\sigma_P(x,\xi)-p(x,x,\xi)$  is bounded from $H^{s,\,\mathcal{L}}(G)$ to $H^{s-(m-\frac 1 \kappa(\rho-(2\kappa-1)\delta)),\,\mathcal{L}}(G)$.
\end{lemma}
\begin{proof}

As in the proof of Lemma \ref{Lemma2} we need to show that
\begin{equation}\label{CV2}
\sigma_P(x,\xi)-p(x,x,\xi)\in S^{m-\frac 1 \kappa(\rho-(2\kappa-1)\delta),\,\mathcal{L}}_{\rho, \max\{\delta,\frac{(\rho+\delta)}{2}\}}=S^{m-\frac 1 \kappa(\rho-(2\kappa-1)\delta),\,\mathcal{L}}_{\rho, \frac{(\rho+\delta)}{2}}.
\end{equation}
 To prove \eqref{CV2} we use the asymptotic expansion in \eqref{asympexp}, which implies that
 
 $$\sigma_P(x,\xi)-p(x,x,\xi)\sim \sum_{|\beta|\geq 1}\partial_Y^{(\beta)}\Delta^\beta_\xi p(x,y,\xi)|_{y=x}.$$
The asymptotic formula above means that, for all $N\in \mathbb{N}_0$,
 $$\sigma_P(x,\xi)-p(x,x,\xi)=\sum_{1\leq|\beta|\leq N}\Delta_\xi^\beta\int\limits_G w_\xi(z)(\partial_Z^{(\beta)} w_\xi(z))\sigma_{A}(z^{-1}x,\xi)dz+r_{N}(x,\xi),$$
 with $r_N(x,\xi)\in S^{m-\frac{(\rho+\delta)}{2}(N+1),\,\mathcal{L}}_{\rho,\frac{(\rho+\delta)}{2}}(G\times \widehat{G})$. Let $N\geq 1$ and define $$S_{N}(x,\xi):=\sigma_P(x,\xi)-p(x,x,\xi)-r_{N}(x,\xi).$$
 Now we expand $\sigma_{A}(z^{-1}x,\xi)$ by using Taylor expansion with respect to $z^{-1}$ at $z^{-1}=e$ as in the proof of Lemma \ref{Lemma1}, and have
 $$S_{N}(x,\xi)=\sum_{1\leq|\beta|\leq N}\Delta_\xi^\beta\int\limits_G w_\xi(z)(\partial_Z^{(\beta)} w_\xi(z))dz \, \sigma_{A}(x,\xi)$$
 $$+\sum_{|\gamma|=1}\sum_{1\leq|\beta|\leq N}\Delta_\xi^\beta\int\limits_G w_\xi(z)(\partial_Z^{(\beta)} w_\xi(z))q_{\gamma}(z) \, \sigma_{A,\gamma}(z^{-1}x,\xi)dz$$
 $$=I(x,\xi)+J(x,\xi),$$ 
where $\sigma_{A,\gamma}\in S^{m+\delta|\gamma|,\,\mathcal{L}}_{\rho,\delta}(G\times\widehat{G})$ ($|\gamma|=1$) comes, as before, from the remainder term of Taylor expansion. To have \eqref{CV2} we need $I,J$ and $r_N$ to belong to the subelliptic class in \eqref{CV2}. We then analyse the three terms separately starting from $I$.

Note that, when $|\beta|=1$, $w_\xi$ and $\partial^{(\beta)}_X w_\xi$ are even and odd respectively (see Proposition 3.11 in \cite{RuzhanskyTurunen2011}), so we have
$$\int\limits_G w_\xi(z)(\partial_Z^{(\beta)} w_\xi(z))dz =0, \quad \text{for}\,\,|\beta|=1,$$ and 
$$I(x,\xi)=\sum_{2\leq|\beta|\leq N}\Delta_\xi^\beta\int\limits_G w_\xi(z)(\partial_Z^{(\beta)} w_\xi(z))dz \, \sigma_{A}(x,\xi).$$
 In particular $I(x,\xi)$ will be given by a sum of terms of the form
 \begin{equation}\label{termI}
     \sum_{2\leq|\beta|\leq N}\int\limits_G (\Delta_\xi^\eta w_\xi(z))(\Delta_\xi^\lambda \partial_Z^{(\beta)} w_\xi(z))dz \, \Delta_\xi^\mu\sigma_{A}(x,\xi),
  \end{equation}
 with $|\eta+\lambda+\mu|\geq |\beta|$. Due to the measure of the support of $w_\xi$, each term of the form \eqref{termI} satisfies 
 
 \begin{equation}\label{est1}
 \left|\int\limits_G (\Delta_\xi^\eta w_\xi(z))(\Delta_\xi^\lambda \partial_Z^{(\beta)} w_\xi(z))dz \right| \leq C \langle \xi\rangle^{-\frac{(\rho+\delta)}{2\kappa}n}
 \langle \xi\rangle^{\frac{n(\rho+\delta)}{4\kappa}-\rho|\eta|}
 \langle \xi\rangle^{\frac{n(\rho+\delta)}{4\kappa}-\rho|\lambda|+\frac{(\rho+\delta)}{2\kappa}|\beta|}.
 \end{equation}
 Then for $I$ we get
 $$\| \widehat{\mathcal{M}}(\xi)^{-m+\frac 1 \kappa(\rho-(2\kappa-1)\delta)} I(x,\xi) \|_{op}\leq C \sum_{2\leq|\beta|\leq N}\,\, \sum_{\lambda,\eta,\mu; |\eta+\lambda+\mu|\geq |\beta|}\| \widehat{\mathcal{M}}(\xi)^{-m+\rho|\mu|}\Delta_\xi^\mu\sigma_{A}(x,\xi) \|_{op}$$
 $$ \times\langle\xi \rangle^{-\rho(|\lambda|+|\eta|)+\frac{(\rho+\delta)}{2\kappa}|\beta|} \| \widehat{\mathcal{M}}(\xi)^{-\rho|\mu|+\frac 1 \kappa(\rho-(2\kappa-1)\delta)} \|_{op}$$
 $$\leq C\sum_{2\leq|\beta|\leq N}\,\, \sum_{\lambda,\eta,\mu; |\eta+\lambda+\mu|\geq |\beta|} \sup_{i=1,...,d_\xi}\langle \nu_{ii}(\xi)\rangle^{-\rho(|\lambda|+|\eta|+|\mu|)+\frac{(\rho+\delta)}{2}|\beta|+\frac 1 \kappa(\rho-(2\kappa-1)\delta)}\lesssim C,$$
 since $$-\rho(|\lambda|+|\eta|+|\mu|)+\frac{(\rho+\delta)}{2}|\beta|=-\rho(|\lambda|+|\eta|+|\mu|-|\beta|) + \frac{\delta-\rho}{2}|\beta|$$
 $$ < \frac{\delta-\rho}{2}|\beta|\leq \delta-\rho<0, $$ and $\delta-\rho+\frac{1}{\kappa}(\rho-(2\kappa-1)\delta)\leq 0$ for $\delta$ such that $\rho-(2\kappa-1)\delta>0$.
 Finally, by using the same kind of estimates as above in combination with Leibniz rule (see Remark \ref{remarkD}), we can also conclude the more general estimate
 $$\| \widehat{\mathcal{M}}(\xi)^{-m+\frac 1 \kappa(\rho-(2\kappa-1)\delta)+|\alpha|\rho+|\beta|\max\{\delta, \frac{(\rho+\delta)}{2} \}}\partial_X^{(\beta)}\triangle_\xi^\alpha I(x,\xi) \|_{op}\leq C$$
 which, in particular, gives that $I(x,\xi)\in S^{m-\frac 1 \kappa(\rho-(2\kappa-1)\delta),\,\mathcal{L}}_{\rho, \max\{\delta,\frac{(\rho+\delta)}{2}\}}=S^{m-\frac 1 \kappa(\rho-(2\kappa-1)\delta),\,\mathcal{L}}_{\rho, \frac{(\rho+\delta)}{2}}$.
 
 We now consider the term $J(x,\xi)$. In this case we have
 
 $$\| \widehat{\mathcal{M}}(\xi)^{-m+\frac 1 \kappa(\rho-(2\kappa-1)\delta)} J(x,\xi) \|_{op}\leq C 
 \sum_{1\leq|\beta|\leq N}\,\, \sum_{\lambda,\eta,\mu; |\eta+\lambda+\mu|\geq |\beta|}\sum_{\gamma, |\gamma|=1}
 \| \widehat{\mathcal{M}}(\xi)^{-m+\frac 1 \kappa(\rho-(2\kappa-1)\delta)}$$
 $$ \int\limits_G (\Delta_\xi^\eta w_\xi(z))(\Delta_\xi^\lambda \partial_Z^{(\beta)} w_\xi(z))q_\gamma(z) \, \Delta_\xi^\mu\sigma_{A,\gamma}(z^{-1}x,\xi)dz\|_{op},
 $$
 so, as for $I(x,\xi)$, we prove that each term in the sum is bounded.
 Arguing as before (and using the fact that $(\Delta_\xi^\eta w_\xi(z))(\Delta_\xi^\lambda \partial_Z^{(\beta)} w_\xi(z))$ commutes with $\Delta_\xi^\mu\sigma_{A,\gamma}(z^{-1}x,\xi)$) we get 
 $$\| \widehat{\mathcal{M}}(\xi)^{-m+\frac 1 \kappa(\rho-(2\kappa-1)\delta)}
 \int\limits_G (\Delta_\xi^\eta w_\xi(z))(\Delta_\xi^\lambda \partial_Z^{(\beta)} w_\xi(z))q_\gamma(z) \Delta_\xi^\mu\sigma_{A,\gamma}(z^{-1}x,\xi)dz\|_{op}$$
 $$\leq C \sup_{i=1,...,d_\xi}\langle \nu_{ii}(\xi)\rangle^{-\frac{(\rho+\delta)}{2\kappa}-\rho(|\eta|+|\lambda|+|\mu|)+\frac{(\rho+\delta)}{2}|\beta|+\frac 1 \kappa(\rho-(2\kappa-1)\delta)}$$
 $$\leq C \sup_{i=1,...,d_\xi}\langle \nu_{ii}(\xi)\rangle^{-\frac{(\rho+\delta)}{2\kappa}-\rho(|\eta|+|\lambda|+|\mu|-|\beta|)+\frac{(\delta-\rho)}{2}|\beta|+\frac 1 \kappa(\rho-(2\kappa-1)\delta)}\leq C.$$
The more general estimate (for $|\alpha|,|\beta|\neq 0$)
$$\| \widehat{\mathcal{M}}(\xi)^{-m+\frac 1 \kappa(\rho-(2\kappa-1)\delta)+|\alpha|\rho+|\beta|\max\{\delta, \frac{(\rho+\delta)}{2} \}}\partial_X^{(\beta)}\triangle_\xi^\alpha J(x,\xi) \|_{op}\leq C$$
giving that $J\in S^{m-\frac 1 \kappa(\rho-(2\kappa-1)\delta),\,\mathcal{L}}_{\rho, \max\{\delta,\frac{(\rho+\delta)}{2}\}}(G\times \widehat{G})=S^{m-\frac 1 \kappa(\rho-(2\kappa-1)\delta),\,\mathcal{L}}_{\rho, \frac{(\rho+\delta)}{2}}(G\times \widehat{G})$ can be derived by following the same steps as in the proof of Lemma \ref{Lemma1}, therefore we omit the proof.

We are now left with the study of $r_N$. In this case, by the properties of the remainder in the asymptotic formula, we have that, for all $N\in \mathbb{N}$, $r_N\in S^{m-\frac{(\rho+\delta)}{2}(N+1),\,\mathcal{L}}_{\rho,\frac{(\rho+\delta)}{2}}(G\times \widehat{G})$, for all $0\leq \delta<\rho\leq 1$ such that $\rho-(2\kappa-1)\delta>0$.
This shows \eqref{CV2} and concludes the proof.
\end{proof}

 \begin{proof}[Proof of Theorem \ref{MainTheorem}] Let us consider the amplitude operator $P$ in Proposition \ref{P}. In view of Remark \ref{SerenaRemark2}, it suffices to prove that  $$Q: H^{\frac{m-\frac 1 \kappa(\rho-(\kappa-1)\delta)}{2},\,\mathcal{L}}(G)\rightarrow H^{\frac{-(m-\frac 1 \kappa(\rho-(\kappa-1)\delta))}{2},\,\mathcal{L}}(G)$$ is a bounded operator.
To show this property it is enough to observe that
 $$\sigma_Q(x,\xi):=(\sigma_{A}(x,\xi)-p(x,x,\xi))+(p(x,x,\xi)-\sigma_{P}(x,\xi)),$$
 therefore, by Lemma \ref{Lemma1} and Lemma \ref{Lemma2}, the required boundedness of $Q$ follows.
This completes the proof of the subelliptic sharp G\r{a}rding inequality.
\end{proof}

\section{Final remarks}\label{FinalRemarks}
Due to the statement of our main result some remarks are in order.

Considering the general version of the Euclidean sharp G\r{a}rding inequality for $(\rho,\delta)$-classes, and its correspondent in the compact Lie group setting given by Theorem \ref{MainTheorem} in the case $\kappa=1$, that is Corollary \ref{Corolario}, one would expect that the suitable version of the subelliptic sharp G\r{a}rding inequality should be as follows:
\medskip

\noindent{\textbf{Conjecture. }}{\it Let $G$ be a compact Lie group and let $\mathcal{L}=\mathcal{L}_X$ be the (positive) sub-Laplacian associated with a system  $X=\{X_{i}\}_{i=1}^{k}$ of left-invariant vector fields  satisfying H\"ormander's condition of step $\kappa$. For  $0\leq \delta<\rho\leq 1$ and $m\in \mathbb{R}$, let $A\equiv a(x,D):C^\infty(G)\rightarrow\mathscr{D}'(G)$ be a continuous linear operator with symbol  $a\in {S}^{m,\,\mathcal{L}}_{\rho,\delta}( G\times \widehat{G})$. Then, if $a(x,[\xi])\geq 0$ for all $(x,[\xi])\in G\times \widehat{G}$, there exists a positive constant $C$ such that
\begin{equation}\label{Remark:SGIT}
    \mathsf{Re}(Au,u)\geq -C\Vert u\Vert_{H^{\frac{m-(\rho-\delta)}{2},\,\mathcal{L}}(G)}^2,
\end{equation}
for all $u\in C^{\infty}(G)$.}
\medskip

The subelliptic setting and the noncommutativity of the group have not allowed us to prove the expected version \eqref{Remark:SGIT} of the subelliptic sharp G\r{a}rding inequality. However, it is worth to make some considerations to stress the importance and nontriviality of our result.

Let us first recall that some inclusions between global subelliptic classes and standard global classes can be established. Specifically, one has that
\begin{itemize}
    \item[1.] If $m>0$, $0\leq\delta<\frac{\rho}{\kappa}$, and $\rho\leq 1$,  then $S^{m,\mathcal{L}}_{\rho,\delta}\subset \mathscr{S}^{m}_{\frac{\rho}{\kappa},\delta}(G)$.
    \item[2.] If $m\leq0$, $0\leq\delta<\frac{\rho}{k}$, and $\rho\leq 1$,  then $S^{m,\mathcal{L}}_{\rho,\delta}\subset \mathscr{S}^{\frac{m}{\kappa}}_{\frac{\rho}{k},\delta}(G)$.
\end{itemize}
Three immediate consequences of the inclusions above and of Corollary \ref{Corolario} are described below.
\medskip

{\it Consequence 1.} Let $m>\frac{\rho}{\kappa}-\delta>0$, $0\leq\delta<\frac{\rho}{\kappa}\leq \frac{1}{\kappa}$, and $a\in S^{m,\,\mathcal{L}}_{\rho,\delta}( G\times \widehat{G})$. Then, if $a(x,[\xi])\geq 0$ for all $(x,[\xi])\in G\times \widehat{G}$, there exists a positive constant $C$ such that
\begin{equation}\label{Rem1}
    \mathsf{Re}(Au,u)\geq -C\Vert u\Vert_{H^{\frac{m-(\frac{\rho}{\kappa}-\delta)}{2}}(G)}^2\geq -C\Vert u\Vert_{H^{\frac{\kappa m-(\frac{\rho}{\kappa}-\delta)}{2},\,\mathcal{L}}(G)}^2,
\end{equation}
for all $u\in C^{\infty}(G)$.
\medskip

{\it Consequence 2.} Let $0<m\leq \frac{\rho}{\kappa}-\delta$, $0\leq\delta<\frac{\rho}{\kappa}\leq \frac{1}{\kappa}$, and $a\in S^{m,\,\mathcal{L}}_{\rho,\delta}( G\times \widehat{G})$. Then, if $a(x,[\xi])\geq 0$ for all $(x,[\xi])\in G\times \widehat{G}$, there exists a positive constant $C$ such that
\begin{equation}\label{Rem2}
    \mathsf{Re}(Au,u)\geq -C\Vert u\Vert_{H^{\frac{m-(\frac{\rho}{\kappa}-\delta)}{2}}(G)}^2\geq -C\Vert u\Vert_{H^{\frac{m-(\frac{\rho}{\kappa}-\delta)}{2},\,\mathcal{L}}(G)}^2,
\end{equation}
for all $u\in C^{\infty}(G)$.
\medskip

{\it Consequence 3.} Let $m\leq 0$, $0\leq\delta<\frac{\rho}{\kappa}\leq \frac{1}{\kappa}$, and $a\in S^{m,\,\mathcal{L}}_{\rho,\delta}( G\times \widehat{G})$. Then, if $a(x,[\xi])\geq 0$ for all $(x,[\xi])\in G\times \widehat{G}$, there exists a positive constant $C$ such that
\begin{equation}\label{Rem3}
    \mathsf{Re}(Au,u)\geq -C\Vert u\Vert_{H^{\frac{\frac m \kappa -(\frac{\rho}{\kappa}-\delta)}{2}}(G)}^2\geq -C\Vert u\Vert_{H^{\frac{\frac m \kappa-(\frac{\rho}{\kappa}-\delta)}{2},\,\mathcal{L}}(G)}^2,
\end{equation}
for all $u\in C^{\infty}(G)$.
\medskip

The previous inequalities, namely \eqref{Rem1}, \eqref{Rem2} and \eqref{Rem3}, show that Theorem \ref{MainTheorem} for subelliptic classes ($\kappa\neq 1$) does not follow from Theorem \ref{MainTheorem} applied to standard global classes ($\kappa=1$), that is, from our Corollary \ref{Corolario}. In other words, the subelliptic sharp G\r{a}rding inequality is not a consequence of the so called {\it elliptic} sharp G\r{a}rding inequality.  Note also that Theorem \ref{MainTheorem} applied to subelliptic classes gives results better than those in \eqref{Rem1}, \eqref{Rem2} and \eqref{Rem3}.

To conclude, let us say that it is natural to conjecture that \eqref{Remark:SGIT} holds true in the subelliptic setting. Of course one can find particular operators for which \eqref{Remark:SGIT} is satisfied. An immediate example is given by operators of the form $A=a(x)\mathcal{L}$, where $a(x)$ is a nonnegative smooth function and $\mathcal{L}$ is a positive sub-Laplacian. However, this is a very specific operator for which the validity of \eqref{Remark:SGIT} follows from simple and direct computations.
For the sake of completeness we will briefly prove this fact below.
\medskip

Note that we can write
\begin{align*}
\mathsf{Re}(a(x)\mathcal{L}u,u)&=\mathsf{Re}\big(\big[M_a,\mathcal{L}^{1/2}\big]\mathcal{L}^{1/2}u,u\big)+\mathsf{Re}(\mathcal{L}^{1/2}a\mathcal{L}^{1/2}u,u),\\
&=\big(\big[M_a,\mathcal{L}^{1/2}\big]\mathcal{L}^{1/2}u,u\big)+\|\sqrt{a}\mathcal{L}^{1/2}u\|^2_{L^2(G)},
\end{align*}
where $\big[M_a,\mathcal{L}^{1/2}\big]$ stands for the commutator between the multiplicative operator $M_af:=a f$ and $\mathcal{L}^{1/2}$. In this particular case (it is not true in general in our non Euclidean setting), we have that $\big[M_a,\mathcal{L}^{1/2}\big]$ is of subelliptic order $0$, namely its symbol belongs to the class $S^{0,\mathcal{L}}_{1,0}(G\times\widehat{G})$. This allows us to estimate $\mathsf{Re}\big(\big[M_a,\mathcal{L}^{1/2}\big]\mathcal{L}^{1/2}u,u\big)=
\mathsf{Re}\big(\mathcal{L}^{-1/4}\big[M_a,\mathcal{L}^{1/2}\big]\mathcal{L}^{1/2}u,\mathcal{L}^{1/4}u\big)
\leq C \|u\|_{1/2,\mathcal{L}}^2$, for some $C>0$, and to conclude that
$$\mathsf{Re}(a(x)\mathcal{L}u,u)\geq -C \|u\|_{1/2,\mathcal{L}}^2,\quad \forall u\in C^\infty (G),$$
where the index $\frac 1 2=\frac{m-1}{2}$ ($m=2$) is ``optimal''.
\medskip
    
Of course, by using the steps above, we have that \eqref{Remark:SGIT} holds true for any operator of the form $A=a(x)\mathcal{L}^{m}$ with $m\in \mathbb{R}$.

That said, it is important to underline once more that the noncommutative structure plays a determinant role in the validity of fundamental a priori estimates such as the one studied here, therefore it would not be surprising if the expected optimal result in \eqref{Remark:SGIT} can be attained only under very particular conditions and/or with tools still to be developed.

\bibliographystyle{amsplain}

\end{document}